\documentclass[11pt]{article}

\usepackage[top=1 in, bottom=1 in, left=1 in, right = 1 in]{geometry}
\usepackage{amsmath}
\usepackage{amsthm}
\usepackage{amssymb}
\usepackage{setspace}
\usepackage{graphicx}
\usepackage{amsbsy}
\usepackage{xcolor}
\usepackage{hyperref}

\newcommand{\pnorm}[3]{\| #1 \|_{L^{#2}(#3)}}
\newcommand{\WW}{\textcolor{black}}

\newcommand{\partder}[2]{\frac{\partial #1}{ \partial #2}}

\newcommand{\abs}[1]{\left| #1  \right|}

\newcommand{\heat}[1]{(\partial_t - d_{#1}\Delta) #1}

\newcommand{\R}{\mathcal{R}}

\newcommand{\F}{\mathcal F}
\newcommand{\V}{\mathcal V}
\newcommand{\tot}{\text{total}}

\newcommand{\eps}{\varepsilon}

\def\XXint#1#2#3{{\setbox0=\hbox{$#1{#2#3}{\int}$ }
		\vcenter{\hbox{$#2#3$ }}\kern-.6\wd0}}

\title{Analysis of a Reaction-Diffusion SIR Epidemic Model with Noncompliant Behavior\thanks{Submitted to the editors \today.}}

\author{Christian Parkinson\thanks{Department of Mathematics, University of Arizona, Tucson, AZ
  (\email{chparkin@math.arizona.edu, weinanwang@math.arizona.edu}).}
\and Weinan Wang\footnotemark[2]}

\usepackage{amsopn}

\usepackage[english]{babel}
\usepackage{cleveref}
\newtheorem{theorem}{Theorem}[section]
\newtheorem{lemma}[theorem]{Lemma}
\begin{document}

\begin{center}
\Large Analysis of a Reaction-Diffusion SIR Epidemic Model with Noncompliant Behavior
\end{center}

\begin{center}
\large Christian Parkinson\footnote{Department of Mathematics, University of Arizona, 617 N. Santa Rita Ave, Tucson, AZ, 85741 \\(chparkin@math.arizona.edu)} and Weinan Wang\footnote{Department of Mathematics, University of Oklahoma, 601 Elm Ave, Norman, OK, 73019 \ (ww@ou.edu)}
\end{center}

% REQUIRED
\noindent{\bf Abstract.}
Recent work from public health experts suggests that incorporating human behavior is crucial in faithfully modeling an epidemic. We present a reaction-diffusion partial differential equation SIR-type population model for an epidemic including behavioral concerns. In our model, the disease spreads via mass action, as is customary in compartmental models. However, drawing from social contagion theory, we assume that as the disease spreads and prevention measures are enacted, noncompliance with prevention measures also spreads throughout the population. We prove global existence of classical solutions of our model, and then perform $\R_0$-type analysis and determine asymptotic behavior of the model in different parameter regimes. Finally, we simulate the model and discuss the new facets which distinguish our model from basic SIR-type models. \\

% REQUIRED

% REQUIRED

\section{Introduction}

In the early days of the COVID-19 pandemic, many nations implemented intervention methods in attempt to slow the spread of the disease. In the absence of vaccines, some common intervention strategies are mandated social-distancing, mask wearing, and self-quarantine. However, studies by public health experts suggest that nontrivial portions of the population may not comply with prevention measures like these, and that this noncompliance affects the spread of the disease \cite{med1,med2,med3,med4,med5}. Accordingly, there has been recent interest in incorporating human behavior and noncompliance with governmental interventions into mathematical models of epidemiology \cite{CP1,CP2,ModelWithHumanBehavior1,ModelWithHumanBehavior2,ModelWithHumanBehavior3}.

Social contagion theory hypothesizes that behaviors, attitudes, sentiments, and mental states can spread within social groups analogously to the manner in which a disease spreads \cite{SC1,SC2,SC3}. Among other things, social contagion theory has been used to model adolescent sexual behavior \cite{SC4}, illicit drug use \cite{SC5}, depression \cite{SC6,SC7}, and violent crime \cite{SC8}. 

Motivated by social contagion theory, the authors of \cite{CP1} present a SIR-type ordinary differential equation model for epidemics wherein governmental protocols are enacted, but noncompliant behavior evolves as a parallel disease at the same time as the actual disease. In this paper, we propose a similar model using reaction-diffusion partial differential equations. Several authors have analyzed diffusive extensions of basic ordinary differential equation epidemic models \cite{RD1,RD2,RD3,RD4,RD5}. However, to the authors' knowledge, none of these have attempted to incorporate human behavioral effects. 

A basic reaction-diffusion epidemic model considers three subpopulations: the susceptible population ($S$), the infected or infectious population ($I$), and the recovered or removed population ($R$). The populations evolve in an open, bounded, connected domain $\Omega \subset \mathbb R^n$ with Lipschitz boundary (typically $n=2$, though this is not usually important for the analysis). We assume a constant birth rate $b \in C(\overline \Omega)$ into the susceptible population, and that deaths occur proportional to population sizes with rate $\delta > 0$. Finally, if populations diffuse at constant rates $d_S,d_I,d_R>0$ and infection spreads via a nonlinear mass-action term, we arrive at simple reaction-diffusion epidemic model: \begin{equation}\label{eq:basicSIR}\begin{split}
	\heat S &= b(x) -\beta SI - \delta S,\\
	\heat I &= \beta SI - (\gamma+\delta) I, \\
	\heat R &= \gamma I - \delta R, 
	\end{split}\,\,\,\,\,\,\, (x,t) \in \Omega \times (0,\infty) \end{equation} where $\beta > 0$ is the infection rate and $\gamma > 0$ is the recovery rate. This is typically accompanied by zero-flux boundary conditions \begin{equation} \nabla S \cdot n = \nabla I \cdot n = \nabla R \cdot n = 0,\,\,\, \text{ for }  (x,t) \in \partial \Omega \times (0,\infty),
	\end{equation} and nonnegative initial conditions $S(\cdot,0)=S_0,I(\cdot,0)=I_0,R(\cdot ,0)=R_0$ where $S_0,I_0,R_0 \in C(\overline \Omega).$

 We append \eqref{eq:basicSIR} by assuming prevention measures have been implemented to slow the spread of the disease, and including populations $S^*,I^*, R^*$ consisting of those who do not comply with the measures. Henceforth, the asterisk will denote noncompliance with prevention measures, while populations $S,I,R$ with no asterisk denote compliant populations. In particular, we let $N^* = S^*+I^*+R^*$ denote the total noncompliant population. Among compliant populations, we reduce the infectivity by a factor of $\alpha \in [0,1]$. Again, infection spreads via nonlinear mass-action terms, but to account for the reduction in infectivity, in any such terms, $S$ and $I$ are replaced with $(1-\alpha)S$ and $(1-\alpha)I$ respectively. Treating noncompliance as a social contagion, we include additional mass-action terms which facilitate transfer from compliant to noncompliant behavior with ``infectivity" rate $\mu>0$. Likewise, we assume noncompliant populations become compliant proportional to the sizes of the populations with rate $\nu \ge 0.$  With these assumptions, we propose the model
	\begin{equation} \label{eq:SIRwithCompliance} \begin{split}
	\heat S &= \xi b(x) -\beta (1-\alpha)S((1-\alpha)I + I^*) - \mu SN^* + \nu S^*  - \delta S,\\
	\heat I &= \beta (1-\alpha)S((1-\alpha)I + I^*) - \gamma I - \mu IN^* + \nu I^* - \delta I, \\
	\heat R &= \gamma I - \mu RN^* + \nu R^*- \delta R,\\
	\heat {S^*} &=(1-\xi)b(x) -\beta S^*((1-\alpha)I + I^*) + \mu SN^* - \nu S^* - \delta S^*,\\
	\heat {I^*} &= \beta S^*((1-\alpha)I + I^*) - \gamma I^* +\mu IN^* - \nu I^* - \delta I^*, \\
	\heat {R^*}&= \gamma I^* + \mu RN^* - \nu R^* - \delta R^*,
	\end{split}
	\end{equation} where, to reiterate, $N^* = S^* + I^* + R^*$ is the total noncompliant population. One final parameter is $\xi \in [0,1]$: the portion of the newly introduced susceptible population which is compliant. For several of the results below, we will be interested in the cases $\xi = 1$ or $\xi = 0$ so that alternately everyone is born compliant or everyone is born noncompliant. As with the basic model, we consider zero-flux boundary conditions and continuous initial data. This paper is devoted to analysis of \eqref{eq:SIRwithCompliance} in different parameter regimes. To contextualize the analysis, we briefly mention results regarding \eqref{eq:basicSIR}.

	The rest of the paper is organized as follows. In Section \ref{s2}, we introduce preliminaries and notations. In Section \ref{s3}, we state our main results on the global existence of solutions. In Section \ref{s4}, we investigate the basic reproduction number $\mathcal R_0$ and its relation to asymptotic behavior of the model. In Section \ref{s5}, we present and discuss some simulations of the model. We conclude with a brief discussion of our results and avenues for future work in section \ref{sec:conclusion}.
	
	\section{Preliminaries}\label{s2}
	
	  In this section, we prove some basic results and establish some lemmas that will be useful moving forward.
	
	Under the assumptions listed above (specifically, $\Omega$ is open, bounded, connected with Lipschitz boundary, and zero-flux boundary conditions), one easily establishes bounds on the total population as long as classical solutions exist. Indeed, for \eqref{eq:basicSIR} we define \begin{equation} \label{eq:totalPop} N_{\tot}(t) = \int_{\Omega} (S(x,t) + I(x,t) +R(x,t)) dx \end{equation} or for \eqref{eq:SIRwithCompliance} we define \begin{equation} \label{eq:totalPop2} N_{\tot}(t) = \int_\Omega (S(x,t) + I(x,t) + R(x,t) + S^*(x,t) + I^*(x,t) + R^*(x,t))dx. \end{equation} In either case, one checks that \begin{align*} N_{\tot}'(t) =  \|b\|_{L^1(\Omega)} - \delta N_{\tot}(t), 
	\end{align*} whereupon \begin{equation} \label{eq:massBound} N_{\tot}(t) \le N_{\tot}(0)e^{-\delta t} + \frac{\|b\|_{L^1(\Omega)}}{\delta}, \,\,\,\,\,\, t \ge 0. \end{equation} As we will see shortly, given nonnegative initial conditions, the populations remain nonnegative as long as they exist, so this establishes a bound on the total population: the $L^1$-norm of the sum of all compartments. This will be useful in section \ref{sec:compDFE}. However, to establish global existence, we need stronger $L^\infty$ bounds. 
	
	Supposing that initial profiles are continuous, local existence and continuity of solutions to \eqref{eq:basicSIR} and \eqref{eq:SIRwithCompliance} follows since the nonlinearities are locally Lipschitz \cite[Theorem 11.12]{smoller}. We use some lemmas to prove that as long as solutions exist and remain nonnegative, they remain bounded. This is sufficient for global existence (see, for example, \cite[Lemma 1.1]{pierre}). \\
	
	\begin{lemma} \label{lem:posPreserve} Suppose that $d > 0$ and $u \in C([0,T] ; C^2(\Omega) \cap C(\overline \Omega))$ satifies \begin{equation} \label{eq:lem1} \begin{split} \partial_t u - d\Delta u &= f(x,t,u) \ge 0, \,\,\,\,\,\, (x,t) \in \Omega \times [0,T] \\ u(x,0) &= u_0(x) \ge 0, \,\,\,\,\, x \in \Omega. \end{split} \end{equation} along with the zero-flux boundary condition. Then $u(x,t) \ge 0$ for all $(x,t) \in \Omega \times [0,T].$\end{lemma}
	
	\begin{proof} This lemma follows directly from the comparison principle since the zero function is a subsolution. \end{proof}
	
	\begin{lemma} \label{lem:bound} Suppose that $u \in C([0,T] ; C^2(\Omega) \cap C(\overline \Omega))$ satisfies  \begin{equation} \label{eq:lem2} \begin{split} \partial_t u - d\Delta u &\le f(x,t) , \,\,\,\,\,\, (x,t) \in \Omega \times [0,T] \\ u(x,0) &= u_0(x), \,\,\,\,\, x \in \Omega. \end{split} \end{equation} Then for any $t \ge 0$, $$\|u(t)\|_\infty \le\| u_0\|_{\infty} + \int^t_0 \|f(s)\|_{\infty} ds,$$ where $\|u(t)\|_{\infty} = \sup_{x\in\overline\Omega} \abs{u(x,t)}.$ \end{lemma}
	
	\begin{proof} Define $v(x,t) = \|u_0\|_{\infty} + \int^t_0 \|f(s)\|_\infty ds  - u(x,t)$. Then $$v_t - d\Delta v = \|f(t)\|_{\infty} - (u_t - d\Delta u) \ge 0$$ and $v(x,0) = \|u_0\|_\infty - u_0(x) \ge 0$ so applying lemma \ref{lem:posPreserve} implies that $v$ is nonnegative and the claim follows. \end{proof}
	
	With these alone, we can establish global-in-time existence for \eqref{eq:basicSIR}. This is a classical result. We repeat the proof because it is quite short and provides contrast for the relative difficulty in proving global existence for \eqref{eq:SIRwithCompliance}. 
 
 The basic outline of the argument is as follows. Let $T^*$ be the supremum of all $\tau > 0$ such that classical solutions to \eqref{eq:basicSIR} exist on $[0,\tau)$. First, note that since the initial data $S_0,I_0,R_0$ are nonnegative functions, the solutions remain nonnegative as long as they exist. This follows by quasi-positivity of the right hand side and a comparison principle as in \cite[Theorem 11, p. 29]{kuttler}. Thus by \cite[Lemma 1.1]{pierre}, to prove global existence (i.e. to prove that $T^* = +\infty$), it suffices to prove that if \WW{$T^* < \infty,$} then $\|S(t)\|_\infty, \|I(t)\|_{\infty}, \|R(t)\|_\infty$ are bounded in $[0,T^*)$. We prove this in the next theorem.

	\begin{theorem} \label{thm:globalExistenceBasic} Suppose that $S(\cdot,t),I(\cdot,t),R(\cdot,t)$ are classical solution of \eqref{eq:basicSIR} on $[0,T^*)$ with nonnegative initial data $S_0,I_0,R_0 \in L^\infty(\Omega)$. If $T^*$ is finite, then there exists a constant $M > 0$ (depending on $T^*$) such that $\|S(t)\|_{\infty}, \|I(t)\|_\infty, \|R(t)\|_\infty \le M$  for $ t \in [0,T^*)$. \end{theorem}
	
	\begin{proof}  As stated above, the solution remain nonnegative as long as it exists. In particular, this means that $$\heat S \le b(x) ,$$ whereupon lemma \ref{lem:bound} immediately yields $$\|S(t)\|_\infty \le \|S_0\|_\infty + T^*\|b\|_{L^\infty(\Omega)} := M_S  , \,\,\,\,\,\, t \in [0,T^*).$$ Next, using an integrating factor in the $I$ equation, we have $$\partial_t\big(e^{(\gamma+\delta) t} I \big) - d_I \Delta\big( e^{(\gamma+\delta) t} I \big) = \beta S \big(e^{(\gamma+\delta) t} I \big) \le \beta M_S \big(e^{(\gamma \WW{+\delta}) t} I).$$ 
 Thus by lemma \ref{lem:bound}, $$e^{(\gamma+\delta) t}\|I(t)\|_\infty \le \|I_0\|_{\infty} + \beta M_S \int^t_0  e^{(\gamma+\delta) s} \|I(s)\|_\infty ds.$$ By Gronwall's inequality, we conclude that $$ e^{(\gamma+\delta) t} \|I(t)\|_{\infty} \le \|I_0\| e^{\beta M_S t } \,\,\,\,\, \implies \,\,\,\,\, \|I(t)\|_\infty \le \|I_0\| e^{\WW{\beta M_S T^*}}:= M_I, \,\,\,\,\,\, t \in [0,T^*).$$ Finally, $$\heat R = \gamma I \le \gamma M_I$$  so Lemma 2 gives $$\|R(t)\|_\infty \le \|R_0\|_\infty + \int^t_0 \gamma M_I ds \le \|R_0\|_\infty + \gamma M_I T^* :=M_R, \,\,\,\,\,\, t \in [0,T^*).$$ Thus the theorem holds with $M = \max \{M_S,M_I,M_R\}$. \end{proof}

 With this, we commence with the analysis of \eqref{eq:SIRwithCompliance}, with the goal of establishing global-in-time existence and long time asymptotic behavior in the ensuing sections.

 \section{Global Existence}\label{s3}

    In this section we establish global existence for \eqref{eq:SIRwithCompliance}. To reiterate some of the discussion above, since each piece of initial data is nonnegative and continuous, local existence follows since the nonlinearity is locally Lipschitz. Note that the nonlinearity is still quasi-positive, so non-negativity is preserved as long as solutions exist. Equation \eqref{eq:massBound} also establishes a bound on total population that holds for \eqref{eq:SIRwithCompliance}. 
    
    With these notes, the proof of global existence for \eqref{eq:SIRwithCompliance} is much simpler if the diffusion coefficients do not depend on compliant/noncompliant status. In this case, we can easily adapt the proof of theorem \ref{thm:globalExistenceBasic}. \\
	
	\begin{theorem} Suppose that $S(\cdot,t),I(\cdot,t),R(\cdot,t),S^*(\cdot,t),I^*(\cdot,t),R^*(\cdot,t)$ are classical solutions of \eqref{eq:SIRwithCompliance} on $[0,T^*)$ with nonnegative initial data $S_0,I_0,R_0,S^*_0,I^*_0,R^*_0 \in L^\infty(\Omega)$ and that $d_S = d_{S^*}$, $d_I = d_{I^*}$ and $d_R = d_{R^*}$. If $T^*$ is finite, then there exists a constant $M > 0$ (depending on $T^*$) such that $$\|S(t)\|_{\infty}, \|I(t)\|_\infty, \|R(t)\|_\infty, \|S^*(t)\|_{\infty}, \|I^*(t)\|_\infty, \|R^*(t)\|_\infty\le M$$  for $ t \in [0,T^*)$.  \end{theorem}
	
	\begin{proof} Since $d_X = d_{X^*}$ for $X \in \{S,I,R\}$, we drop the asterisk in the diffusion coefficients. Since nonnegativity is preserved, we have $$0 \le S,S^* \le S+S^*, \,\,\,\, 0 \le I,I^* \le I+I^*, \,\,\,\, 0 \le R,R^* \le R+R^*.$$ Thus it suffices to prove that $$\Sigma = S+S^*, \,\,\,\, \Phi = I+I^*, \,\,\,\, \Upsilon = R+R^*$$ remain bounded. Since the diffusion coefficients are the same, we can add the first and fourth equation of \eqref{eq:SIRwithCompliance} to see that $$(\partial_t - d_S \Delta) \Sigma = b(x) - \beta ((1-\alpha)S + S^*)((1-\alpha)I + I^*) - \delta \Sigma \le b(x) - \delta \Sigma.$$ Using an integrating factor, we have $$(\partial_t  - d_S \Delta )(e^{\delta t}\Sigma) \le b(x)e^{\delta t}$$ so by lemma \ref{lem:bound}, $$e^{\delta t}\|\Sigma(t)\|_\infty \le \|\Sigma_0\|_\infty + \int^t_0 \|b\|_\infty e^{\delta s}ds \le \|S_0\|_\infty + \|S^*_0\|_\infty + \frac{\|b\|_\infty}\delta (e^{\delta t} - 1)$$ and thus $$\|\Sigma(t)\|_\infty \le \|S_0\|_\infty + \|S_0^*\|_\infty + \frac{\|b\|_\infty}{\delta}(1 - e^{-\delta t}) \le \|S_0\|_\infty + \|S_0^*\|_\infty + \frac{\|b\|_\infty}{\delta}:= M_\Sigma.$$ This provides a bound for $\Sigma$ which is actually independent of $T^*$. Next, adding the second and fifth equations of \eqref{eq:SIRwithCompliance}, we have \begin{align*} (\partial_t - d_I \Delta )\Phi &=  \beta ((1-\alpha)S + S^*)((1-\alpha)I + I^*) - \gamma \Phi - \delta \Phi \\&\le \beta \Sigma \Phi - (\gamma + \delta )\Phi  \\ &\le (\beta M_\Sigma - (\gamma +\delta)) \Phi. \end{align*} This leads to $$(\partial_t - d_I \Delta) (\Phi e^{((\gamma + \delta) - \beta M_\Sigma )t}) \le 0$$ so by lemma \ref{lem:bound}, $$e^{ ((\gamma + \delta) - \beta M_\Sigma)t} \|\Phi(t)\|_\infty \le \|\Phi_0\|_\infty \le \|I_0\|_\infty + \|I^*_0\|_\infty$$ and thus $$\|\Phi(t) \|_\infty \le (\|I_0\|_\infty + \|I^*_0\|_\infty)\WW{\max\{1,e^{(\beta M_\Sigma - (\gamma + \delta))T^*}\}} := M_\Phi.$$ Finally, adding the third and sixth equations in \eqref{eq:SIRwithCompliance}, we have $$(\partial_t - d_R \Delta) \Upsilon = \gamma \Phi - \delta \Upsilon$$ and so $$(\partial_t - d_R \Delta) (\Upsilon e^{\delta t}) = \gamma \Phi e^{\delta t} \le \gamma M_\Phi e^{\delta t}$$ which yields $$e^{\delta t} \|\Upsilon(t)\|_\infty \le \|\Upsilon_0\|_\infty + \gamma M_\Phi \int^t_0 e^{\delta s}ds \le \|R_0\|_\infty + \|R^*_0\|_\infty + \frac{\gamma M_\Phi}{\delta}(e^{\delta t} - 1)$$ whereupon $$\|\Upsilon(t)\|_\infty \le  \|R_0\|_\infty + \|R^*_0\|_\infty + \frac{\gamma M_\Phi}{\delta} := M_\Upsilon.$$ This proves the theorem with $M = \max\{M_\Sigma,M_\Phi,M_\Upsilon\}$. \end{proof}

 We include the above result only to emphasize that in certain cases, the behavior of the actual epidemic described by\eqref{eq:SIRwithCompliance} should be somewhat akin to that of \eqref{eq:basicSIR}, under the association $(\Sigma,\Phi,\Upsilon) \leftrightarrow (S,I,R)$ in the respective models. In the case that all diffusion coefficients are different, we can still achieve global existence, though the proof is no longer nearly so elementary, since we can no longer simply add equations to eliminate the second mass-action nonlinearity which describes transmission of noncompliance. In order to establish global existence with arbitrary positive diffusion coefficients, we first prove two more lemmas, the first of which is an $L^p$ version of lemma \ref{lem:bound}.  \\
	
	\begin{lemma} \label{lem:boundLp} \WW{ Fix $f : \Omega\times [0,T] \to \mathbb R$ and suppose that $v: \Omega \times [0,T] \to \mathbb R$ is a solution of \begin{equation} \label{eq:ForcedHeatEq} (\partial_t - d \Delta)v = f \end{equation} for $(x,t) \in \Omega \times (0,T]$ with initial data $v(\cdot,0) = v_0 \in L^\infty(\Omega)$ and $\partder v n = 0$ on $\partial \Omega \times [0,T]$. Then for any $p \in (1,\infty)$ and $t \in (0,T]$, $$\|v(t)\|_{L^p(\Omega)} \le \|v_0\|_{L^p(\Omega)} + \int^t_0 \|f(\tau)\|_{L^p(\Omega)} d\tau.$$} \end{lemma}
	
	\begin{proof} This is proven with elementary \emph{a priori} bounds. We multiply the equation by $p\abs v^{p-2}v$, integrate in space, and use H\"older's inequality to see \begin{equation} \label{eq:LpboundHeateq} \int_\Omega (\partial_t (\abs{v}^p) - dp\abs{v}^{p-2}v\Delta v) dx = p \int_\Omega \abs{v}^{p-2}vf dx \le p \|v(t)\|_{L^p(\Omega)}^{p/q}\|f(t)\|_{L^p(\Omega)}, \end{equation} where $q$ is the dual exponent of $p$. On the left hand side, we integrate by parts in the latter term (and use the zero-flux boundary condition) to see $$\int_\Omega - \abs{v}^{p-2}v\Delta v dx = (p-1) \int_{\Omega} \abs{\nabla v}^2 \abs{v}^{p-2} dx \ge 0. $$ Thus putting $G(t) = \|v(t)\|^p_{L^p(\Omega)}$, equation \eqref{eq:LpboundHeateq} reads $$G'(t) \le p\pnorm{f(t)}p\Omega G(t)^{1/q} \,\,\,\, \implies \,\,\,\, G(t)^{-1/q}G'(t) \le p\pnorm{f(t)}p\Omega.$$ Integrating gives $$\frac{1}{1-\frac 1 q} (G(t)^{1-1/q} - G(0)^{1-1/q}) \le p\int^t_0 \pnorm{f(\tau)}p\Omega d\tau,$$ and since $1-\frac 1 q = \frac 1 p$, we arrive at $$\pnorm{v(t)}p{\Omega} \le \pnorm{v_0}p\Omega + \int^t_0 \pnorm{f(\tau)}p\Omega d\tau$$  as desired.
  \end{proof} 
	
	 \begin{lemma} \label{lem:keyLemma} \WW{Suppose that $v,w : \Omega \times [0,T] \to \mathbb R$ are such that \begin{equation} \label{eq:diffCoeff}(\partial_t - d \Delta)v \le c_1 \partial_t w + c_2 \Delta w \end{equation} for $(x,t) \in \Omega \times (0,T]$ and that $v,w$ have the same bounded initial data and each satisfy a homogeneous Neumann boundary condition. Then for any $p \in (1,\infty)$, there is $C > 0$ such that for all $t \in (0,T]$, $$\|v\|_{L^p(\Omega_t)} \le C(1+\|w\|_{L^p(\Omega_t)}),$$ where $\Omega_t = \Omega \times [0,t),$ and $C$ depends on the ambient parameters as well as the initial-data. }\end{lemma}

  \textit{Remark.} Lemma \ref{lem:keyLemma} is a key lemma for us. Roughly speaking, for a function satisfying a reaction-diffusion equation, this lemma allows us to modify the diffusion coefficient while still maintaining $L^p$ control of the function.

\begin{proof} \WW{The proof is by duality. Fix $p \in (1,\infty)$ and $t \in (0,T]$ and let $q$ be the dual exponent to $p$. For any nonnegative $g \in C^\infty(\Omega_t)$, let $\phi$ be the nonnegative smooth solution of $$-(\partial_t  + d \Delta ) \phi= g,$$ on $\Omega \times [0,t)$ with $\phi(\cdot,t) = 0$ and $\frac{\partial \phi }{\partial n} = 0$ on $\partial \Omega.$ It is classical from the theory of linear parabolic equations (see \cite{Krylov,LadyEtAl,WuYinWang}) that $$\pnorm{\partial_t \phi}{q}{\Omega_t} + \pnorm{\Delta \phi}{q}{\Omega_t} + \sup_{s \in [0,t)} \pnorm{\phi(s)}q\Omega \le C \pnorm g q{\Omega_t}.$$ Multiplying \eqref{eq:diffCoeff} by $\phi$ and integrating gives $$\int_{\Omega_t} v(\underbrace{-\partial_t \phi - d \Delta \phi}_{=g}) \, dx dt +  - \int_{\Omega}  \phi_0 v_0 dx  \le  \int_{\Omega_t} w(-c_1\partial_t \phi + c_2 \Delta \phi) \, dx dt + c_1\int_{\Omega} \phi_0 w_0 dx,$$ where the boundary terms have vanished due to the homogeneous Neumann boundary data, and we have used $\phi(\cdot,t) = 0$. Since the initial data for $v,w$ is the same, we combine those terms to arrive at $$\int_{\Omega_t} vg \,dx dt \le \int_{\Omega_t} w(-c_1\partial_t \phi + c_2 \Delta \phi) \, dx dt + (c_1-1)\int_\Omega \phi_0 v_0 dx.$$ Then applying H\"older's inequality gives \begin{align*}\left|\int_{\Omega_t} vg\, dx dt\right|  &\le   \pnorm{w}{p}{\Omega_t} \left(c_1\pnorm{\partial_t \phi}{q}{\Omega_t} + c_2\pnorm{\Delta \phi}q{\Omega_t}\right)  + C\|w_0\|_{L^p(\Omega)}\|\phi_0\|_{L^1(\Omega)} \\ &\le C(1+\pnorm w p {\Omega_t})\left(c_1\pnorm{\partial_t \phi}{q}{\Omega_t} + c_2\pnorm{\Delta \phi}q{\Omega_t}+ \sup_{s \in[0,t)]} \pnorm{\phi(s)}q\Omega \right) \\  &\le C(1+\pnorm{w}p{\Omega_t}) \pnorm{g}{q}{\Omega_t}\end{align*} Since this holds for all $g \in C^\infty(\Omega),$ we conclude that $$\pnorm v p {\Omega_t} \le C(1+\pnorm w p {\Omega_t})$$ as desired. }\end{proof}

Using these lemmas, we first prove $L^p$ boundedness for the solution $(S,I,R,S^*,I^*,R^*)$ of \eqref{eq:SIRwithCompliance} as long as it exists. While this theorem does not establish global existence, the bulk of the work toward proving global existence is in the proof of this theorem. \\

\begin{theorem} \label{thm:LpBound}Suppose that $S,I,R,S^*,I^*,R^*$ are classical solutions of \eqref{eq:SIRwithCompliance} on $\Omega \times [0,T^*)$ with nonnegative initial data $S_0,S^*_0,I_0,I^*_0,R_0,R^*_0 \in C(\overline \Omega)$ and let $\Omega_t = \Omega \times [0,t)$. If $T^*$ is finite,  then for any $p \in (1,\infty)$, there is $M > 0$ such that for all $t \in [0,T^*)$, $$\pnorm{S}p{\Omega_t}, \pnorm{I}p{\Omega_t},\pnorm{R}p{\Omega_t},\pnorm{S^*}p{\Omega_t},\pnorm{I^*}p{\Omega_t},\pnorm{R^*}p{\Omega_t}\le M.$$ \end{theorem}

\begin{proof} Fix $p \in (1,\infty)$ and $t \in [0,T^*)$. In the course of the proof, we will invent auxiliary functions $z_X$ for  $X \in \{S,I,R,S^*,I^*,R^*\}$. For all such functions, we assume homogeneous initial data $z_X(\cdot,0) = 0$, and zero-flux boundary data. Also $C$ will be a positive constant that changes from line to line and depends on the data including $T^*$.

Let $z_S$ be the solution of $$(\partial_t - d_S \Delta)z_S =  \xi b(x) + \nu S^*.$$ Note that $$(\partial_t-d_S \Delta)(S-z_S) =  -\beta (1-\alpha)S((1-\alpha)I + I^*) - \mu SN^*-\delta S^* \le 0$$ from which lemma \ref{lem:boundLp} yields $$\pnorm{S(t)-z_S(t)}p{\Omega} \le \pnorm{S_0}p{\Omega} \,\,\,\, \implies \,\,\,\,\, \pnorm{S(t)}p\Omega \le C(1 + \pnorm{z_S(t)}p\Omega).$$ Taking the $p^{\text{th}}$ power and integrating in time gives \begin{equation} \label{eq:Spnorm}\pnorm{S}p{\Omega_t}^p \le C(1+\pnorm{z_S}p{\Omega_t}^p).\end{equation} Now let $z_I$ be the solution of $$(\partial_t - d_I\Delta)z_I = \xi b(x) + \nu S^* + \nu I^*.$$  Then $$(\partial_t - d_S)S + (\partial_t - d_I\Delta)(I-z_I) = -\mu(S+I)N^* -\gamma I- \delta(S+I) \le 0$$ so that $$(\partial_t - d_I\Delta)(I-z_I) \le -(\partial_t - d_S\Delta )S.$$ Hence, applying lemma \ref{lem:keyLemma}, we have $$\|I-z_I\|_{L^p(\Omega_t)} \le C(1+\pnorm S p {\Omega_t})$$ and thus \eqref{eq:Spnorm} leads to \begin{equation}\label{eq:Ipnorm}
\|I\|^p_{L^p(\Omega_t)} \le C(1+\pnorm{z_S}p{\Omega_t}^p + \pnorm{z_I}p{\Omega_t}^p).
\end{equation} Next, let $z_R$ be the solution of $$(\partial_t - d_R \Delta)z_R = \xi b(x) + \nu S^* + \nu I^* + \nu R^*.$$ Then $$(\partial_t - d_S\Delta )S + (\partial_t - d_I\Delta)I + (\partial_t - d_R\Delta)(R-z_R) = -\mu(S+I+R)N^* - \delta(S+I+R) \le 0$$ and so $$(\partial_t - d_R\Delta)(R-z_R) \le -(\partial_t - d_S\Delta)S - (\partial_t - d_I\Delta)I.$$ From here, an obvious extension of lemma \ref{lem:keyLemma} (to allow for multiple functions on the right hand side) yields $$\pnorm{R-z_R}p{\Omega_t} \le C(1+\pnorm S p {\Omega_t} + \pnorm I p {\Omega_t}).$$ Then \eqref{eq:Spnorm} and \eqref{eq:Ipnorm} lead to \begin{equation} \label{eq:Rpnorm} \pnorm{R}p{\Omega_t}^p \le C(1 + \pnorm{z_S}p{\Omega_t}^p + \pnorm{z_I}p{\Omega_t}^p + \pnorm{z_R}p{\Omega_t}^p). \end{equation} Continuing, define $z_{S^*}$ to be the solution of $$(\partial_t - d_{S^*}\Delta)z_{S^*} = b(x) + \nu I^* + \nu R^*.$$ Then \begin{align*}(\partial_t - d_S)S +& (\partial_t - d_I\Delta)I + (\partial_t - d_R\Delta)R + (\partial_t - d_{S^*}\Delta)(S^*-z_{S^*}) \\ &= -\beta S^*((1-\alpha)I+I^*) - \mu(I+R)N^* - \delta(S+I+R) \le 0.\end{align*} Thus $$(\partial_t - d_{S^*}\Delta)(S^*-z_{S^*}) \le -[\heat S + \heat I + \heat R] $$ whereupon $$\pnorm{S^*-z_{S^*}}p{\Omega_t} \le C(1+ \pnorm S p {\Omega_t}+\pnorm I p {\Omega_t}+\pnorm R p {\Omega_t}).$$ Equations \eqref{eq:Spnorm}-\eqref{eq:Rpnorm} then give \begin{equation} \label{eq:Sspnorm}\pnorm{S^*}p{\Omega_t}^p \le C(1 +  \pnorm{z_S}p{\Omega_t}^p + \pnorm{z_I}p{\Omega_t}^p + \pnorm{z_R}p{\Omega_t}^p + \pnorm{z_{S^*}}p{\Omega_t}^p).\end{equation} Next we let $z_{I^*}$ solve $$(\partial_t -d_{I^*}\Delta)z_{I^*} = b(x) + \nu R^*,$$ so that \begin{align*} \heat S + \heat I &+ \heat R + \heat {S^*} + (\partial_t - d_{I^*}\Delta)(I^* - z_{I^*}) \\ &= - \mu RN^* -\delta(S+I+R+S^*+I^*) \le 0. \end{align*} Hence, applying lemma \ref{lem:keyLemma} and equations \eqref{eq:Spnorm}-\eqref{eq:Sspnorm}, we arrive at \begin{equation} \label{eq:Ispnorm}\pnorm{I^*}p{\Omega_t}^p \le C(1 +  \pnorm{z_S}p{\Omega_t}^p + \pnorm{z_I}p{\Omega_t}^p + \pnorm{z_R}p{\Omega_t}^p + \pnorm{z_{S^*}}p{\Omega_t}^p + \pnorm{z_{I^*}}p{\Omega_t}^p).\end{equation} Lastly, if $z_{R^*}$ satisfies $$(\partial_t - d_{R^*}\Delta)z_{R^*} = b(x),$$ then \begin{equation*} \begin{split} &\heat S + \heat I + \heat R + \\   &\heat {S^*} + \heat{I^*} + (\partial_t - d_{R^*}\Delta)(R^*-z_{R^*}) \end{split} = -\delta(S+I+R+S^*+I^*+R^*) \le 0,\end{equation*} so that lemma \ref{lem:keyLemma} and equations \eqref{eq:Spnorm}-\eqref{eq:Ispnorm} lead to \begin{equation}  \label{eq:Rspnorm} \pnorm{R^*}p{\Omega_t}^p \le C \left( \begin{split} 1 +&\pnorm{z_S}p{\Omega_t}^p  + \pnorm{z_I}p{\Omega_t}^p + \pnorm{z_R}p{\Omega_t}^p \\ &+ \pnorm{z_{S^*}}p{\Omega_t}^p + \pnorm{z_{I^*}}p{\Omega_t}^p + \pnorm{z_{R^*}}p{\Omega_t}^p \end{split} \right).  \end{equation} 

Now we need to ``close the loop" by establishing bounds on the auxiliary functions $z_X$ in terms of our original functions. To do so, define $$Z(t) = \pnorm{z_S}p{\Omega_t}^p + \pnorm{z_I}p{\Omega_t}^p + \pnorm{z_R}p{\Omega_t}^p + \pnorm{z_{S^*}}p{\Omega_t}^p + \pnorm{z_{I^*}}p{\Omega_t}^p + \pnorm{z_{R^*}}p{\Omega_t}^p$$ and $$P(t) = \pnorm{S}p{\Omega_t}^p+ \pnorm{I}p{\Omega_t}^p +\pnorm{R}p{\Omega_t}^p +\pnorm{S^*}p{\Omega_t}^p +\pnorm{I^*}p{\Omega_t}^p +\pnorm{R^*}p{\Omega_t}^p.$$ Then equations \eqref{eq:Spnorm}-\eqref{eq:Rspnorm} show that \begin{equation} \label{eq:Nineq}P(t) \le C(1+Z(t)). \end{equation} On the other hand, each of the functions $z_X$ satisfies an equation of the form $$(\partial_t - d_X\Delta)z_X = Cb(x) + \ell(S,I,R,S^*,I^*,R^*)$$ where $\ell$ is a linear function of $(S,I,R,S^*,I^*,R^*)$. Thus by lemma \ref{lem:boundLp}, each $z_X$ satisfies a bound of the form \begin{equation*} \begin{split} \pnorm{z_X(t)}p\Omega \le C\bigg(1 + \int^t_0 \big( &\pnorm{S(\tau)}p{\Omega}+ \pnorm{I(\tau)}p{\Omega} +\pnorm{R(\tau)}p{\Omega} \\ &+ \pnorm{S^*(\tau)}p{\Omega} +\pnorm{I^*(\tau)}p{\Omega} +\pnorm{R^*(\tau)}p{\Omega} \big) d\tau\bigg). \end{split} \end{equation*} Taking the $p^{\text{th}}$ power, using Jensen's inequality, and integrating in $t$ yields $$\pnorm{z_X}p{\Omega_t} \le C\left(1+\int^t_0 P(s)ds\right)$$ whereupon summing all these bounds gives \begin{equation}\label{eq:Zineq} Z(t) \le C\left(1+\int^t_0P(s)ds\right).\end{equation} Inserting \eqref{eq:Zineq} into \eqref{eq:Nineq}, we see that $$P(t) \le C\left(1+\int^t_0P(s)ds\right).$$ Finally, an application of Gr\"onwall's inequality shows that $P(t)$ remains bounded, and thus the $p$-norms of each of $S,I,R,S^*,I^*,R^*$ remain bounded. \end{proof}
	
	We note in particular that this result holds for every $p \in (1,\infty)$. Using this and a classical result on parabolic regularity, we can prove boundedness of solutions to \eqref{eq:SIRwithCompliance} using the Sobolev embedding theorem. \\
	
	\begin{theorem} Suppose that $S(\cdot,t),I(\cdot,t),R(\cdot,t),S^*(\cdot,t),I^*(\cdot,t),R^*(\cdot,t)$ are classical solutions of \eqref{eq:SIRwithCompliance} on $[0,T^*)$ with nonnegative initial data $S_0,I_0,R_0,S^*_0,I^*_0,R^*_0 \in C(\overline \Omega)$. If $T^*$ is finite, then there exists a constant $M > 0$ (depending on $T^*$) such that $$\|S\|_{L^\infty(\Omega_t)}, \|I\|_{L^\infty(\Omega_t)}, \|R\|_{L^\infty(\Omega_t)}, \|S^*\|_{L^\infty(\Omega_t)}, \|I^*\|_{L^\infty(\Omega_t)}, \|R^*\|_{L^\infty(\Omega_t)}\le M$$  for all $t \in [0,T^*)$. In particular, this implies global existence of classical solutions of \eqref{eq:SIRwithCompliance}. \end{theorem}
	
	\begin{proof} Using results on parabolic regularity (see \cite[Chap. 2, Sec. 4]{Krylov}, \cite[Chap. 4, Sec. 9]{LadyEtAl}, \cite[Chap. 9, Sec. 2]{WuYinWang}), for any $X \in \{S,I,R,S^*,I^*,R^*\}$, we can achieve the following bound for any $t \in [0,T^*]$: \begin{equation} \label{eq:gradBounds1}
\|\partial_t X\|_{L^p(\Omega_t)} + \|\nabla X\|_{L^p(\Omega_t)} \le C(\|X_0\|_{L^p(\Omega)} + \|F(S,I,R,S^*,I^*,R^*)\|_{L^p(\Omega_t)})
\end{equation} where $F$ is the right hand side of the corresponding equation. In particular, $F$ satisfies \begin{align*}\abs{F(S,I,R,S^*,I^*,R^*)} \le C(1+ &S+I+R+S^*+I^*+R^*\\&+S^2 + I^2 + R^2 + (S^*)^2 + (I^*)^2 + (R^*)^2),\end{align*} so that \begin{align*}\|F(S,I,R,S^*,I^*,R^*)\|_{L^p(\Omega_t)} \le C(1 + &\pnorm{S}{p}{\Omega_t}  +  \pnorm{S}{2p}{\Omega_t}^2\\ &+ \pnorm{I}{p}{\Omega_t}  +  \pnorm{I}{2p}{\Omega_t}^2 +\cdots).\end{align*} Thus theorem \ref{thm:LpBound} and equation \eqref{eq:gradBounds1} show that each of $S,I,R,S^*,I^*,R^*$ remain bounded in $W^{1,p}(\Omega_t)$ uniformly in $t \in [0,T^*)$ for any $p \in (1,\infty)$. Taking $p$ large enough, the Sobolev embedding theorem provides the same bounds in $L^\infty(\Omega_t)$. \end{proof}

\textit{Remark.} In fact, similar methods are used to achieve global existence for systems of parabolic equations in \cite{morgan}, where the author achieves bounds which are uniform in time, as opposed to the above bounds which depend on $T$. We will need uniform bounds on certain populations when we prove theorem \ref{t.w11191} later. However, because we would like bounds which are partially quantitative, we derive the bounds using Green's functions. \\

 With this, we move on to analysis of long time behavior and stability of steady state solutions. 

		\section{Basic reproduction number and stability of the disease-free steady states}\label{s4}

    In this section, we would like to establish long term behavior of \eqref{eq:SIRwithCompliance} in different parameter regimes. Following \cite{WZ}, for the steady-state analysis, we define $u = (u_1,u_2,u_3,u_4,u_5,u_6) = (I,I^*,S,S^*,R,R^*)$ and rewrite \eqref{eq:SIRwithCompliance} in the form \begin{equation} \label{eq:newVersionSIR} 
        (\partial_t - D\Delta)u = \mathcal F(x,u) - \mathcal V^{-}(x,u) + \mathcal V^+(x,u)
        \end{equation} where $D$ is a diagonal matrix containing the diffusion coefficients, and the functions $\mathcal F, \mathcal V^-, \mathcal V^+$ account for (respectively) the introduction of new infections into compartments, the transfer out of compartments due to infection, death, recovery, or infection with noncompliance, and the transfer into compartments due to birth or infection with noncompliance. Specifically, for our model,  \begin{equation}
		\mathcal F(x,u)=
		\begin{pmatrix}
		\beta (1-\alpha)u_3((1-\alpha)u_1 + u_2) \\
		\beta u_4((1-\alpha)u_1 + u_2)\\
		0\\
		0\\
		0\\
		0
		\end{pmatrix}
		=:
		\begin{pmatrix}
		\F_1(x,u)\\
		\F_2(x,u)\\
		0\\
		0\\
		0\\
		0
		\end{pmatrix},
		\end{equation}
		
		\begin{equation}
		\V^-(x,u)=
		\begin{pmatrix}
		(\gamma+\delta)u_1 + \mu u_1(u_2 + u_4 + u_6)\\
		(\gamma+\delta+\nu)u_2\\
		 \beta (1-\alpha)u_3((1-\alpha)u_1 + u_2) +\mu u_3(u_2+u_4+u_6) + \delta u_3\\
		\beta u_4((1-\alpha)u_1 + u_2)+ (\nu + \delta) u_4\\
	    \mu u_5(u_2+u_4+u_6) +\delta u_5\\
		(\nu + \delta)u_6
		\end{pmatrix} =:
		\begin{pmatrix}
		\V^-_1(x,u)\\
		\V^-_2(x,u)\\
		\V^-_3(x,u)\\
		\V^-_4(x,u)\\
		\V^-_5(x,u)\\
		\V^-_6(x,u)
		\end{pmatrix},
		\end{equation}
  and 
    \begin{equation}
		\V^+(x,u)=
		\begin{pmatrix}
		\nu u_2\\
		\mu u_1 (u_2+u_4+u_6)\\
		\xi b(x) + \nu u_4\\
		(1-\xi)b(x) + \mu u_3(u_2+u_4+u_6)\\
		\gamma u_1 + \nu u_6\\
		\gamma u_2 + \mu u_5(u_2+u_4+u_6)
		\end{pmatrix} =:
		\begin{pmatrix}
		\V^+_1(x,u)\\
		\V^+_2(x,u)\\
		\V^+_3(x,u)\\
		\V^+_4(x,u)\\
		\V^+_5(x,u)\\
		\V^+_6(x,u)
		\end{pmatrix}.
		\end{equation}
        
        We reiterate that $(u_1,u_2,u_3,u_4,u_5,u_6) = (I,I^*,S,S^*,R,R^*)$ so that the equations are rearranged from \eqref{eq:SIRwithCompliance} to include the infected compartments first. Also, since $\mathcal V^-$ is subtracted from the right hand side in \eqref{eq:newVersionSIR}, each of $\mathcal F, \mathcal V^-, \mathcal V^+$ are componentwise nonnegative functions. In what follows, inequalities with vectors will always be interpreted componentwise.
        
        With this \cite{WZ} provide a general framework for local asymptotic stability of disease-free steady state solutions; those in which $I=I^*=0$ (or $u_1 = u_2 = 0$ in the new notation). Using the framework of \cite{WZ}, we define $U_s$ to be the set of all disease-free states:
        \begin{equation}
        U_s:= \{u\geq 0: u_i=0,  i=1,2\}.
        \end{equation}
		The following assumptions from \cite{WZ} are easily verified for our system. 
		\begin{itemize}\label{a.w12191}
        \item [(A1)] For $i = 1,2$, functions $\F_i (x, u), \V_i^+(x, u)$, $\V_i^-(x, u)$ are nonnegative, continuous and continuously differential with respect to $u$.
         \item [(A2)] If $u_i =0$, then $\V_i^- =0$, for $i=1,\ldots,6$. In particular, if $u\in U_s$, then $\V_i^- =0$ for $i=1,2$. 
         \item [(A3)] $\F_i =0$ for $i>2$.
         \item [(A4)] If $u \in U_s$ (so that $u_1=u_2=0$), then $\F_i =\V_i^+ =0$ for $i=1,2$.
\end{itemize}

Along with these four general properties regarding the system, one must consider two more properties regarding linearization about the specific steady state one wishes to analyze. Specifically, suppose that $\tilde u(x)$ is a disease-free steady state solution of \eqref{eq:newVersionSIR}; that is, \begin{equation} \label{eq:newVersonSteady}
-D \Delta \tilde u = \mathcal F(x,\tilde u) - \mathcal V^{-}(x,\tilde u) + \mathcal V^+(x,\tilde u).
\end{equation} \WW{Now, once again following the notation of \cite{WZ}, we linearize \eqref{eq:newVersionSIR} around the disease-free steady state $\tilde u(x)$. Due to assumptions (A2)-(A4), since $\tilde u$ takes values in $U_s$, we have {\renewcommand{\arraystretch}{1.5} $$ [D_u(\V^+-\V^-)](x,\tilde u) = \left[\begin{array}{r|c} \tilde V(x) & 0  \\ \hline \tilde J(x) & -\tilde M(x) \end{array}\right].$$} Thus we arrive at the linearized system \begin{equation}\label{eq:linearizedVersion} \begin{split}
(\partial_t - D_{1:2} \Delta)u_{1:2} &= (\tilde F(x) - \tilde V(x)) u_{1:2}, \\
(\partial_t - D_{3:6}\Delta) u_{3:6} &= -\tilde J(x) u_{1:2} + \tilde M(x)u_{3:6}.
\end{split}
\end{equation} Here $\tilde F$ and $\tilde V$ are $2\times 2$ matrices corresponding to transfer in and out of the infected compartments respectively, $\tilde M$ is a $4\times 4$ matrix corresponding to movement within non-infected compartments, and $\tilde J(x)$ is a $4 \times 2$ matrix corresponding to transfer from non-infected to the infected compartments. The matrix $\tilde J(x)$ turns out to be less important for the analysis since its effects are also captured by $\tilde F(x)$.} Specifically, \begin{equation} \label{eq:Vmat}\tilde V(x) = \left[ \partder{\V^-_i(x,\tilde u(x))}{u_j} - \partder{\V^+_i(x,\tilde u(x))}{u_j}\right]_{1\le i,j \le 2}\end{equation} and \begin{equation} \label{eq:Mmat}
\tilde M(x) = \left[ \partder{\V^+_i(x,\tilde u(x))}{u_j} - \partder{\V^-_i(x,\tilde u(x))}{u_j}\right]_{3\le i,j \le 6}.
\end{equation} With all this, the final two assumptions are \begin{itemize}
\item[(A5)] $\tilde M(x)$ is cooperative, and all eigenvalues of $\tilde M(x)$ have negative real part,
\item[(A6)] $-\tilde V(x)$ is cooperative, and all eigenvalues of $-\tilde V(x)$ have negative real part.
 \end{itemize} Recall, a matrix is called \textit{cooperative} if all off-diagonal elements are nonnegative. 

Assuming we can verify that disease-free steady states for our system verify these properties, we can invoke \cite[Theorem 3.1]{WZ} to prove local asymptotic stability under the further condition that the basic reproductive number $\mathcal R_0$ corresponding to the steady state (which we will define in the ensuing subsections) is sufficiently small. With this, we move on to consider disease-free steady states for our system. The local stability analysis fits in to the general framework presented above. However, the global stability analysis is quite delicate, relying nontrivially on different parameter values. 
 
		\subsection{Noncompliant Disease-Free Equilibrium} \label{sec:nonCompDFE}
		In this section, we consider stability analysis of the  disease-free model and determine the reproduction number $\R^*_0$ in the case that all new individuals introduced are noncompliant with prevention measures. That is, we set $\xi = 0$, so that our model reads
		\begin{equation} \label{e.w09141}
		 \begin{split}
		\heat S &= -\beta (1-\alpha)S((1-\alpha)I + I^*) - \mu SN^* + \nu S^*  - \delta S,\\
		\heat I &= \beta (1-\alpha)S((1-\alpha)I + I^*) - \gamma I - \mu IN^* + \nu I^* - \delta I, \\
		\heat R &= \gamma I - \mu RN^* + \nu R^*- \delta R,\\
		\heat {S^*} &= b(x) -\beta S^*((1-\alpha)I + I^*) + \mu SN^* - \nu S^* - \delta S^*,\\
		\heat {I^*} &= \beta S^*((1-\alpha)I + I^*) - \gamma I^* +\mu IN^* - \nu I^* - \delta I^*, \\
		\heat {R^*}&= \gamma I^* + \mu RN^* - \nu R^* - \delta R^*,
		\end{split}
		\end{equation}
		 Since we are considering a disease-free equilibrium, we set $S=I=R=I^{*}=R^{*}= 0$ in equation \eqref{e.w09141} to arrive at
		\begin{equation}\label{e.w07291} \begin{split}
		(\partial_t- d_{S^*}\Delta )S^* &=b(x) - (\nu+\delta ) S^*, \\
		\frac{\partial S^*}{\partial n}&=0~\text{on}~\partial \Omega. \end{split}
		\end{equation}
		We let $\tilde S^{*}(x)$ be the unique, positive, steady-state solution of \eqref{e.w07291}.  We call $\tilde E^{*}=(0,0,0,\tilde S^{*},0,0)$ the noncompliant disease-free state of \eqref{eq:SIRwithCompliance}. Linearizing the $I^*$ equation in \eqref{e.w09141} around $\tilde E^{*}$ yields
		\begin{equation} \label{eq:IstarEq}
		\begin{split}
		(\partial_t - d_I^* \Delta) I^* 
		&=
		(\beta \tilde S^*- (\gamma   +\nu + \delta)) I^*,\\
		\frac{\partial I^*}{\partial n}&=0~\text{on}~\partial \Omega.
        \end{split} 
		\end{equation}
		Note that if we linearize the $I$ equation from \eqref{e.w09141} around $\tilde E^*$, we immediately see exponential decay of $I$, so in the linearized regime, the only appearance of new infections comes from \eqref{eq:IstarEq}. To arrive at the eigenvalue problem for \eqref{eq:IstarEq}, we use the ansatz $I^*(x,t)=e^{\lambda t}\varphi^*(x)$ and we have
		\begin{equation}\label{e.w07292}
		\begin{split}
		\lambda \varphi^*
		=
		d_I^* \Delta \varphi^*
		+
		(\beta \WW{\tilde S^*}- (\gamma   +\nu + \delta)) \varphi^*,
		\end{split}
		\end{equation}
		By the Krein-Rutman theorem, equation \eqref{e.w07292} has a principal eigenvalue $\lambda^{*}(\tilde S^*)$ given by the variational formula
		\begin{equation}\label{e.w07293}
		\begin{split}
		\lambda^{*}(\WW{\tilde S^*})
		&=-\inf \left\{\int_\Omega \left(d_{I^*} |\nabla \varphi^*|^2 + \left((\gamma   +\nu + \delta) - \beta \tilde S^* \right) |\varphi^*|^2 \right)\,dx: \varphi^* \in H^1(\Omega),
		~\int_\Omega |\varphi^*|^2\,dx=1 \right\}.
		\end{split}
		\end{equation} Intuitively, the sign of this eigenvalue will determine whether $I^*$ is locally increasing or decreasing in time. For standard SIR-type analysis, this same property is often phrased in terms of the basic reproduction number $\R_0^*$ which we define by
		\begin{equation}\label{e.w07294}
		\begin{split}
		\R_0^*
		=
		\sup_{0 \neq \varphi^* \in H^1(\Omega)} \left\{\frac{\int_\Omega \beta \tilde S^* |\varphi^*|^2 \,dx}{\int_\Omega d_{I^*} |\nabla \varphi^*|^2 + (\gamma   +\nu + \delta) |\varphi^*|^2 \,dx }  \right\}
		\end{split}
		\end{equation}
        The classical interpretation of $\R^*_0$ is that it represents the average number of new infections which result from a single infection near the outset of the epidemic. Accordingly, we expect the total number of infections to increase if $\R^*_0 > 1$ and decrease if $\R^*_0 < 1$, so given the above comment regarding $\lambda^*(\tilde S^*)$, we expect some connection between the size of $\R^*_0$ relative to $1$ and the sign of $\lambda^*(\tilde S^*).$ Indeed, this connection is provided by \cite[Theorem 3.1]{WZ}. 

        As stated above, assumptions (A1)-(A4) are easily seen to hold for our system, regardless of the value of $\xi$. Linearizing around $\tilde E^*$ and writing in the notation of \eqref{eq:linearizedVersion}, we have $$
		\tilde M^*(x)=
		\begin{pmatrix}
		-\mu \tilde S^{*} -\delta &\nu &0&0\\
		\mu \tilde S^{*} &-\nu - \delta &0&0\\
		0 &0&-\mu \tilde S^{*}-\delta&\nu\\
		0 &0 &\mu \tilde S^{*}&-\nu - \delta, \\
		\end{pmatrix}
		\,\,\,\, \text{ and } \,\,\,\,
		\tilde V^*(x)=
		\begin{pmatrix}
		\gamma+\delta+\mu \tilde S^{*} & -\nu  \\
		-\mu \tilde S^{*}& \gamma+\delta+\nu \\
		\end{pmatrix}.
		$$ These matrices satisfy (A5) and (A6): both $\tilde M(x)$ and $-\tilde V(x)$ are clearly cooperative, and their eigenvalues have negative real part by Gershgorin's theorem since they are (columnwise) diagonally dominant with negative diagonal entries. Thus we have the following local stabiliy result as a consequence of \cite[Theorem 3.1]{WZ}.
  
		\begin{lemma} \label{lem:localNonComp} As defined above, 
		$\R^*_0-1$ has the same sign as $\lambda^*(\tilde S^*)$. Furthermore, if $\R^*_0 < 1$, then $\tilde E^{*}$ is locally asymptotically stable. 
		\end{lemma}
		
		We emphasize that this is only a local result, in the sense that if the solution $(S,I,R,S^*,I^*,R^*)$ begins near enough to $\tilde E^* = (0,0,0,\tilde S^*,0,0)$, then it will return to $\tilde E^*$. We also note that, by all appearances, this result does not depend significantly on the assumption $\xi = 0$ (meaning that any new members of the population are noncompliant). \WW{However, this assumption is necessary in achieving a more quantitative description of the leading eigenvalue $\lambda^*(\tilde S^*)$ in \eqref{e.w07293} and the reproductive ratio $\mathcal R^*_0$ in \eqref{e.w07294}. If $\xi \in (0,1)$, the equilibrium solution is a system of coupled nonlinear elliptic equations, whereupon it is more difficult to determine conditions for convergence to steady state. We comment more on this in section \ref{sec:conclusion}.}
  
  %guaranteeing existence of the steady-state solution $\tilde S$. If $\xi \in (0,1)$, then the steady-state disease-free system is a pair of coupled nonlinear elliptic equations involving $S,S^*$ for which, to the authors' knowledge, we cannot guarantee existence of solutions. 
		
		Next, we state our main result concerning the stability of $\tilde E^{*}=(0,0,0, \tilde S^{*},0,0)$. For global stability analysis, we assume that the diffusion coefficients do not depend on compliant status so that $d_X= d_{X^*}$ for $X \in \{S,I,R\}$. We also assume that $\nu = 0$, meaning that individuals who become noncompliant will remain noncompliant for all ensuing time. The first assumption allows us to control the nonlinear growth due to noncompliance by analyzing the sums $S+S^*$, $I+I^*$, and $R+R^*$ as in the proof of \ref{thm:globalExistenceBasic}. The latter ensures that the system does not stray from the equilibrium due to large portions of the populations becoming compliant. Within the context of this model, the population becoming compliant again would actually be helpful in slowing the progression of the disease, so this can be seen as a \emph{worst case scenario} assumption.
		\begin{theorem} \label{thm:globalStabilityNonComp}
			Under the conditions that $d_X= d_{X^*}$ for $X \in \{S,I,R\}$ and $\nu = 0$, the following statements hold regarding \eqref{e.w09141}.
			\begin{itemize}
				\item[(i)] If $\mathcal R_0^* <1$, then the disease-free steady state $\tilde E^{*}=(0,0,0,\tilde S^{*},0,0)$ is globally asymptotically stable.
				\item[(ii)] If $\mathcal R_0^* >1$, then there exists a constant $\epsilon_0 >0$ such that any positive solution of \eqref{e.w09141} satisfies
				\begin{equation}
				\limsup_{t\rightarrow \infty}
				\|(S,I,R,S^*,I^*,R^*)-(0,0,0,\tilde S^{*},0,0)\|_{L^\infty(\Omega)}>\epsilon_0.
				\end{equation}
			\end{itemize}
		\end{theorem}
	\begin{proof}

 First note that $$(\partial_t - d_S \Delta)S = -\beta (1-\alpha)S((1-\alpha)I+I^*) - \mu SN^* -\delta S \le -\delta S.$$ Upon using an integrating factor, lemma \ref{lem:bound} immediately yields $$\|S(t)\|_{L^\infty(\Omega)} \le \|S(0)\|_{L^\infty(\Omega)} e^{-\delta t}$$ so that $S\to 0$ uniformly as $t\to \infty$. With this, define $\Sigma^* = S^* - \tilde S^*$, the difference between $S^*(x,t)$ and the steady state solution $\tilde S^*(x)$. Then $$(\partial_t - d_S \Delta)\Sigma^* = -\beta S^*((1-\alpha)I+I^*)  - \delta \Sigma^* + \mu SN^*.$$ Adding the $S$ equation from \eqref{e.w09141} then yields \begin{align*}(\partial_t - d_S \Delta)(\Sigma^* + S) &= -\beta((1-\alpha)S+S^*)((1-\alpha)I+I^*) - \delta(\Sigma^* + S) \\&\le -\delta(\Sigma^* + S)\end{align*} whereupon, by the same reasoning as above, $$\|(\Sigma^*+S)(t)\|_{L^\infty(\Omega)} \le \|(\Sigma^*+S)(0)\|_{L^\infty(\Omega)}e^{-\delta t}$$ so that $\Sigma^* + S \to 0$ uniformly as $t\to \infty$. But then we use$$\|\Sigma^*(t)\|_{L^\infty(\Omega)} \le \|(\Sigma^* + S)(t)\|_{L^\infty(\Omega)} + \|S(t)\|_{L^\infty(\Omega)}$$ to see that $\Sigma^*$ also converges uniformly to zero as $t\to \infty$; that is $S^* \to \tilde S^*$ uniformly as $t\to \infty$. 
	
	In particular, since $S^* \to \tilde S^*$ and $S\to 0$ uniformly, for any $\eps > 0$, we can find $\tau> 0$ such that for all $x \in \Omega$, \begin{equation}\label{eq:Sbounds} S^*(x,t)\le \tilde S^*(x) +\eps \,\,\,\, \text{ and } \,\,\,\, S(x,t)\le \eps \,\,\,\, \text{ for } t \ge \tau. \end{equation}
	
	Next consider $\Phi = I + I^*$. For $t \ge \tau$, this function satisfies \begin{align*}(\partial_t - d_I \Delta) \Phi &= \beta ((1-\alpha)S + S^*)((1-\alpha)I+I^*) - (\delta + \gamma) (I+I^*) \\
	&=(\beta(1-\alpha)^2 S + \beta(1-\alpha)S^* - (\delta+\gamma))I + (\beta (1-\alpha)S + \beta S^* - (\delta + \gamma))I^*\\
	&\le  (\beta (1-\alpha)S + \beta S^* - (\delta + \gamma))\Phi\\
	&\le (\beta (1-\alpha)\eps + \beta(\tilde S^*+ \eps) - (\delta + \gamma))\Phi.
	\end{align*}
	Using an integrating factor we see that \begin{equation}\label{eq:PhiEqBound} (\partial_t - d_I \Delta)[\Phi e^{-\beta(1-\alpha)\eps}t] \le  (\beta (\tilde S^*+\eps)- (\delta+\gamma))[\Phi e^{-\beta(1-\alpha)\eps t}].\end{equation}
	Note that \eqref{eq:PhiEqBound} is akin to \eqref{eq:IstarEq}, but with the extra arbitrarily small $\eps$. By lemma \ref{lem:localNonComp}, $\R^*_0 < 1$ implies that the principle eigenvalue for \eqref{e.w07292} satisfies $\lambda^*(\tilde S^*) < 0$. Reducing $\eps$ if necessary, by continuity, we have $\lambda^*(\tilde S^*+\eps)<0$. Let $\phi_\eps^*$ be the strongly positive eigenfunction corresponding to $\lambda^*(\tilde S^*+\eps)$, and take a constant $A$ large enough that $\Phi(x,\tau)e^{-\beta(1-\alpha)\eps\tau} \le A\phi^*_\eps(x)$. Then by the comparison principle, \begin{equation} \label{eq:PhiSupBound}\Phi(x,t)e^{-\beta(1-\alpha)\eps t} \le A \phi^*_\eps(x) e^{\lambda^*(S^*_0+\eps) (t-\tau)} \,\,\,\, \implies \,\,\,\, \Phi(x,t) \le \tilde A \phi^*_\eps(x) e^{(\lambda^*(\tilde S^*+\eps)+\beta(1-\alpha)\eps)  t}, \,\,\,\,\,\, t \ge \tau. \end{equation} Again, reducing $\eps$ if necessary, we can achieve $\lambda^*(\tilde S^*+\eps)+\beta(1-\alpha)\eps<0$ whereupon \eqref{eq:PhiSupBound} shows that $I,I^* \to 0$ uniformly as $t\to\infty$.  
	Using this result, the equation for $\Upsilon = R+R^*$ is asymptotic to $$(\partial_t - d_R\Delta)\Upsilon = - \delta \Upsilon$$ which implies that $\Upsilon \to 0$ uniformly as well. Thus $(S,I,R,S^*,I^*R^*) \to (0,0,0,S^*_0,0,0)$ uniformly as $t\to \infty$. This proves \textit{(i)}.
	
	Next we prove \textit{(ii)}. Let $\mathcal R^*_0 > 1$ and assume toward a contradiction that for any $\eps_0 > 0$, there is a positive solution of \eqref{e.w09141} satisfying $$\limsup_{t\to\infty}\|(S,I,R,S^*,I^*,R^*)-(0,0,0,S_{0}^{*},0,0)\|\le\epsilon_0.$$ In this case, for any $\eps_0 > 0$, we can find $\tau > 0$ such that \begin{equation} \label{eq:limsupBound} 0 <  \abs{S(x,t)},\abs{I(x,t)}, \abs{I^*(x,t)}, \abs{R(x,t)}, \abs{R^*(x,t)} \le \eps_0, \,\,\,\,\, \text{ for all } x \in \Omega,  \,\, t \ge \tau.\end{equation} Then for $t \ge \tau$, $\Sigma^*(x,t) = S^*(x,t) - S^*_0(x)$ satisfies $$(\partial_t - d_S\Delta)\Sigma^* = -\beta((1-\alpha)S+S^*)-\delta \Sigma^* + \mu SN^* \le (\eps_0-\delta) \Sigma^* + 2\mu\eps_0^2.$$ This yields $$\| \Sigma^*(t)\|_{\infty} \le e^{(\eps_0 - \delta)t}\|\Sigma^*(0)\|_\infty + 2\mu \eps_0^2 t e^{(\eps_0 - \delta)t}, \,\,\,\,\,\,\, t \ge \tau.$$ However, since this holds for any fixed $\eps_0 > 0$, we can take $\eps_0 < \delta$ and this shows that $S^*(x,t) \to \tilde S^*(x)$ uniformly as $t \to \infty$. In particular, increasing $\tau$ if necessary, we have $$S^*(x,t) \ge \tilde S^*(x) - \eps_0, \,\,\,\,\,\ t \ge \tau.$$ But then for $t \ge \tau$, $$(\partial_t - d_I\Delta)I^* \ge \beta((\tilde S^* - \eps_0) - (\gamma + \delta))I^*.$$ Since $\mathcal R^*_0 > 1$, the principle eigenvalue for \eqref{e.w07292} satisfies $\lambda^*(\tilde S^*)>0$ and thus by continuity, there is $\eps_0>0$ small enough that $\lambda^*(\tilde S^*-\eps_0)>0$. Fixing this $\eps_0$ (and the corresponding $\tau > 0$), we let $\phi_{\eps_0}^*(x)$ be the positive solution of \eqref{e.w07292} corresponding to $\lambda^*(\tilde S^*-\eps_0)$, and take $\eta>0$ small enough that $I^*(x,\tau) \ge \eta \phi_{\eps_0}^*(x)$ for all $x \in \Omega$. Then by the comparison principle, $$I^*(x,t) \ge \eta \phi^*_{\eps_0}(x) e^{\lambda^*(\tilde S^*-\eps_0)(t-\tau)}$$ proving that $I^*(x,t)$ grows without bound as $t \to \infty$ which contradicts \eqref{eq:limsupBound}. The contradiction implies that when $\mathcal R^*_0>1$, there is $\eps_0 > 0$ such that any positive solution of \eqref{e.w09141} satisfies $$\limsup_{t\to\infty}\|(S,I,R,S^*,I^*,R^*)-(0,0,0,S_{0}^{*},0,0)\| >\epsilon_0$$ as desired.
	\end{proof}

 \textit{Remark.} We include a brief interpretation of Theorem \ref{thm:globalStabilityNonComp}. The interpretation of result \textit{(i)} is fairly straightforward: under the condition $\R^*_0 < 1$, the disease dies out as time increases. Result \textit{(ii)} establishes a condition under which $(S,I,R,S^*,I^*,R^*) \not\to (0,0,0,\tilde S^*,0,0)$ as time increases. Since the proof is by contradiction, it does not specify which of $(S,I,R,S^*,I^*,R^*)$ fails to converge to its corresponding equilibrium value. However, with some additional reasoning, it is easy to see that one of $\|I(t)\|_\infty,\|I^*\|_{\infty}$ does not tend to zero. Indeed, if both of these tended to zero, then reasoning as in the proof of \textit{(i)}, we would have $\|S(t)\|_{\infty} \to 0, \|(S^*-\tilde S^*)(t)\|_\infty, \|R(t)\|_{\infty}, \|R^*(t)\|_{\infty} \to 0$ whereupon we would return to the equilibrium point, which cannot occur. Thus \textit{(ii)} implies that when $\R^*_0 > 1$, the disease persists in the sense that $I+I^* \not\to 0$ for large time. 
 
%	\begin{proof}
%		First, we re-write $I$ and $I^*$ equations
%		\begin{equation}
%		(\partial_t -d_I \Delta)I
%		=
%		[\beta(1-\alpha)S - (\gamma+\delta)]I + [\beta(1-\alpha)S + \nu]I^* - \mu IN^{*}.
%		\end{equation}
%		and
%		\begin{equation}
%		(\partial_t -d_I^{*} \Delta)I^*
%		=
%		[\beta S_{0}^{*}-(\gamma+\nu+\delta)]I^* +\beta(1-\alpha)S_{0}^{*}I + \mu IN^*.
%		\end{equation}
%		Since we assume that $d_I=d_{I^*}$, we add the two above equations yielding
%		\begin{equation}
%		\begin{split}
%		(\partial_t -d_I \Delta)(I+I^*)
%		&=
%		[\beta(1-\alpha)(S+S_{0}^{*}) - (\gamma+\delta)]I 
%		\\&\quad+ [\beta S_{0}^{*}-(\gamma+\nu+\delta)+\beta(1-\alpha)S + \nu]I^*
%		\\&\leq
%		[\beta(1-\alpha)(\epsilon+S_{0}^{*}) - (\gamma+\delta)]I 
%		\\&\quad+ [\beta S_{0}^{*}-(\gamma+\nu+\delta)+\beta(1-\alpha)\epsilon + \nu]I^*
%		.
%		\end{split}
%		\end{equation}
%		Choose the parameters such that
%		\begin{equation}
%		\beta S_{0}^{*}-(\gamma+\nu+\delta)+\beta(1-\alpha)\epsilon + \nu<0,
%		\end{equation}
%		which is
%		\begin{equation}
%	    0<S_{0}^{*}
%	    <
%	    -(\gamma+\delta)/\beta+(1-\alpha)\epsilon.
%		\end{equation}
%	\end{proof}

		\subsection{Compliant Disease-Free Equilibrium} \label{sec:compDFE}
		In this section, we consider an equilibrium in which the entire population is compliant with prevention measures, and in which all newly introduced members are compliant ($\xi = 1$). That is, we consider the system of equations \begin{equation} \label{eq:allCompliant}
		 \begin{split}
		\heat S &= b(x)-\beta (1-\alpha)S((1-\alpha)I + I^*) - \mu SN^* + \nu S^*  - \delta S,\\
		\heat I &= \beta (1-\alpha)S((1-\alpha)I + I^*) - \gamma I - \mu IN^* + \nu I^* - \delta I, \\
		\heat R &= \gamma I - \mu RN^* + \nu R^*- \delta R,\\
		\heat {S^*} &=-\beta S^*((1-\alpha)I + I^*) + \mu SN^* - \nu S^* - \delta S^*,\\
		\heat {I^*} &= \beta S^*((1-\alpha)I + I^*) - \gamma I^* +\mu IN^* - \nu I^* - \delta I^*, \\
		\heat {R^*}&= \gamma I^* + \mu RN^* - \nu R^* - \delta R^*,
		\end{split}
		\end{equation}

        As above, we would like to establish stability for a disease-free equilibrium solution of this model. Specifically, setting $I=R=S^*=I^*=R^*=0$, we arrive at the equation \begin{equation} \label{eq:compliantSeq}
        \begin{split}
		(\partial_t- d_{S}\Delta )S &=b(x) - \delta S, \\
		\frac{\partial S}{\partial n}&=0~\text{on}~\partial \Omega. \end{split}
        \end{equation} Let $\tilde S(x)$ be the unique, positive, steady-state solution of \eqref{eq:compliantSeq}, so that $\tilde E = (\tilde S,0,0,0,0,0)$ is the disease (and noncompliance) free equilibrium solution of \eqref{eq:allCompliant}. Linearizing the $I$ equation from \eqref{eq:allCompliant} around $\tilde E$ gives
		\begin{equation} \label{eq:IeqLin}
		\begin{split}
		(\partial_t - d_I \Delta) I &=
		(\beta(1-\alpha)^{2}\tilde S - (\gamma + \delta)) I + \nu I^*,\\
        \frac{\partial I}{\partial n}&=0~\text{on}~\partial \Omega. 
		\end{split}
		\end{equation} In this case, the linearization of the $I^*$ equation from \eqref{eq:allCompliant} around $\tilde E$ yields exponential decay, so \eqref{eq:IeqLin} is asymptotic to \begin{equation} \label{eq:IeqLin2}
		\begin{split}
		(\partial_t - d_I \Delta) I &=
		(\beta(1-\alpha)^{2}\tilde S - (\gamma + \delta)) I,\\
        \frac{\partial I}{\partial n}&=0~\text{on}~\partial \Omega. 
		\end{split}
		\end{equation} Accordingly, using the ansatz $I(x,t) = e^{\lambda t}\varphi(x)$, we arrive at the eigenvalue problem \begin{equation}\label{e.w08012}
		\begin{split}
		\lambda \varphi=d_I \Delta \varphi+(\beta(1-\alpha)^{2}\tilde S- (\gamma + \delta)) \varphi.
		\end{split}
		\end{equation}

  By the Krein-Rutman theorem, \eqref{e.w08012} has a principal eigenvalue $\lambda(\tilde S)$ and it is be given by the variational formula
		\begin{equation}\label{e.w08013}
		\begin{split}
		\lambda(\tilde S)
		=-\inf \left\{\int_\Omega \left(d_{I} |\nabla \varphi|^2 + \left((\gamma   + \delta) - \beta(1-\alpha)^{2} \tilde S \right) |\varphi|^2 \right)\,dx: \varphi \in H^1(\Omega),~\int_\Omega |\varphi|^2\,dx=1 \right\}.
		\end{split}
		\end{equation}
		Again, in the linearized regime about $\tilde E$, $I^*$ will decay exponentially, so the only new infections are accounted for by $I$, which we expect to be locally increasing in time if $\lambda(\tilde S) > 0$ and locally decreasing in time if $\lambda(\tilde S) < 0.$ To capture this same behavior in the language of SIR-type models, we define the basic reproduction number 
		\begin{equation}\label{e.w08014}
		\begin{split}
		\R_0
		=
		\sup_{0 \neq \varphi \in H^1(\Omega)} \left\{\frac{\int_\Omega \beta(1-\alpha)^{2} \tilde S |\varphi|^2 \,dx}{\int_\Omega d_{I} |\nabla \varphi|^2 + (\gamma   + \delta) |\varphi|^2 \,dx }  \right\}.
		\end{split}
		\end{equation}

        With all this, we once again have local stability (and the relationship between $\lambda(\tilde S)$ and $\R_0$) as a consequence of \cite[Theorem 3.1]{WZ}. The verification of hypotheses (A1)-(A6) is essentially identical to that presented in section \ref{sec:nonCompDFE}. 

        \begin{lemma} \label{lem:localComp} As defined above, 
		$\R_0-1$ has the same sign as $\lambda(\tilde S)$. Furthermore, if $\R_0 < 1$, then $\tilde E$ is locally asymptotically stable. 
        \end{lemma}

        We would like to establish global stability, but as in section \ref{sec:nonCompDFE}, this is much more delicate. In this case, the most interesting (and complicating) facet of the analysis is the nonlinear growth of the noncompliant populations, which could potentially cause instability of the equilibrium solution $\tilde E = (\tilde S,0,0,0,0,0)$, even in the case that $\R_0 < 1$. If a large portion of the population becomes noncompliant, we are reverted to a situation similar to section \ref{sec:nonCompDFE}, where we have a larger reproductive number $\R_0^*$.  Thus, our global stability result in this case depends on first understanding the total size of the noncompliant population $N^* = S^*+I^*+R^*$ and then ensuring that the noncompliance transmission and recovery parameters $\mu$ and $\nu$ are such that $N^*$ does not grow too rapidly. 

        By \eqref{eq:massBound} and nonnegativity, we have the following $L^1$-bound on $N^*$: \begin{equation} \label{eq:Nmax}
        \|N^*(t)\|_{L^1(\Omega)} \le N_{\tot}(t) \le N_\tot(0)e^{-\delta t} + \frac{\|b\|_{L^1(\Omega)}}{\delta}. 
        \end{equation} Using this bound for $N^*$, we can state and prove our global stability result. 
        
		\begin{theorem}\label{t.w11191}
			 Under the conditions that $d_S = d_I = d_R = d$, $d_{S^*} = d_{I^*} = d_{R^*} = d^*$, the following statements hold.
			\begin{itemize}
				\item[(i)] There is a constant $c > 0$ depending on $d,\delta$ and the domain $\Omega$ such that if $\R_0 < 1$, \begin{equation} \label{eq:munubound}
                    \mu < \frac{\delta^2}{\|b\|_{L^1(\Omega)}}, \,\,\,\,\, \text{ and } \,\,\, \nu > \frac{c\mu\delta \|b\|_{L^1(\Omega)}}{
                    \delta^2 - \mu\|b\|_{L^1(\Omega)}}, 
                \end{equation} then the disease-free steady state $\tilde E=(\tilde S,0,0,0,0,0)$ is globally asymptotically stable for \eqref{eq:allCompliant}.
				\item[(ii)] If $\mathcal R_0 >1$, then there exists a constant $\epsilon_0 >0$ such that any positive solution of \eqref{eq:allCompliant} satisfies
				\begin{equation}
				\limsup_{t\rightarrow \infty}
				\|(S,I,R,S^*,I^*,R^*)-(\tilde S_{0},0,0,0,0,0)\|_{L^\infty(\Omega)}>\epsilon_0.
				\end{equation}
			\end{itemize}
		\end{theorem}

  \textit{Remark.} Before the proof, we reiterate some interpretation, explain the assumptions, and describe the strategy. In this case, there are two manners in which $\tilde E = (\tilde S,0,0,0,0,0)$ could be unstable: (1) the number of infections could grow and persist as will happen when $\R_0 > 1$ in result \textit{(ii)}, or (2) noncompliance could grow in the case that $\mu$---the noncompliance ``infectivity" rate---is too large relative to $\nu$---the noncompliance ``recovery" rate. This latter case could then further be broken down into two types on instability: (2a) the noncompliance could persist, so that $N^* \not \to 0$, or (2b) the growth of the noncompliant population could increase the effective reproductive ratio, meaning that infections surge and perhaps persist, so that $I,I^* \not \to 0.$ For result \textit{(i)}, the assumptions on the diffusion coefficients allow us to focus the first part of the analysis on $N = S+I+R$ and $N^* = S^* + I^* + R^*$. These satisfy an SIS (susceptible-infected-susceptible) system of equations. To circuit the possibility of $N^*$ growing, we use the assumptions about the smallness of $\mu$ and largeness of $\nu$. This will ensure that the population returns to a fully compliant state, whereupon we can use the assumption $\R_0 <1$ to prove that infections die out.  

  \begin{proof}
    To prove \textit{(i)}, define the total compliant population $N = S+I+R$. Because we assume $d_S = d_I = d_R = d$ and $d_{S^*} = d_{I^*} = d_{R^*} = d^*$, from \eqref{eq:allCompliant}, we see that the compliant population $N$ and the noncompliant population $N^* = S^*+I^*+R^*$ satisfy the pair of reaction-diffusion equations \begin{align} 
    \partial_tN &= d\Delta N+ b(x) - (\mu N - \nu) N^* - \delta N, \label{eq:N}\\
    \partial_t N^* &=d^* \Delta N^* +(\mu N - \nu)N^* - \delta N^*. \label{eq:Nstar} 
    \end{align}

    From here, we bound $N$ so that our assumptions regarding $\mu$ and $\nu$ ensure that the first term in \eqref{eq:Nstar} is negative. In doing so, \eqref{eq:Nstar} will yield exponential decay of $N^*$ whereupon arguments similar to those in section \ref{sec:nonCompDFE} will ensure global stability of $\tilde E$.

    To this end, we follow the strategy of \cite{ZAMP,Luo,Ren}. Let $\Gamma(t): C(\overline \Omega) \to C(\overline \Omega)$ denote the $C_0$ semigroup associated with the operator  $(d\Delta-\delta)$ with Neumann boundary conditions. That is, $$(\Gamma(t)\phi)(x) = \int_\Omega G(x,y,t)\phi(y)dy, \,\,\,\, t> 0, \,\, x \in \overline \Omega$$ where $G$ denotes the Green's function corresponding to $(d\Delta-\delta)$ with Neumann boundary conditions on $\partial \Omega$. In particular, there is $M > 0$ such that \begin{equation}\label{eq:GamBound} \|\Gamma(t)\| \le M e^{\alpha t}, \,\,\, t \ge 0 \end{equation} where $\alpha < 0$ is the principle eigenvalue of $(d\Delta - \delta)$ with Neumann boundary conditions. Then for any $t \ge t_0,$ \begin{equation}\label{eq:Nbound}\begin{split} N(x,t) &= \Gamma(t-t_0)N(x,t_0) + \int^t_{t_0} \Gamma(t-s)[b(x)-\mu N(x,s)N^*(x,s) + \nu N^*(x,s)]ds \\ 
    &\le Me^{\alpha(t-t_0)} \|N(\cdot,t_0)\|_{L^\infty(\Omega)} + \int^t_{t_0} \Gamma(t-s)[b(x) + \nu N^*(x,s)]ds \\
    &=Me^{\alpha(t-t_0)} \|N(\cdot,t_0)\|_{L^\infty(\Omega)} + \int^t_{t_0} \int_\Omega G(x,y,t-s)[b(y) + \nu N^*(y,s)]dy ds,
    \end{split} \end{equation} where the bound of the integral follows because $\Gamma(t)$ is strongly positive for $t > 0$ \cite[Corollary 7.2.3]{Smith} and $N,N^*$ are nonnegative. Next, performing a spectral expansion as in \cite{ZAMP,Luo,Ren} and using uniform boundedness of the eigenfunctions of $(d\Delta - \delta)$, one achieves $$G(x,y,t) \le c e^{-\delta t}, \,\,\,\, t > 0,$$ for some constant $c > 0$ which depends on the eigenvalues and eigenfunctions of the operator, which in turn depend on $d,\delta$ and the domain $\Omega$. Applying this in \eqref{eq:Nbound}, we have \begin{equation} \label{eq:Nbound2} \begin{split} N(x,t) &\le Me^{\alpha(t-t_0)} \|N(\dot,t_0)\|_{L^\infty(\Omega)} + c \int^t_{t_0} e^{-\delta(t-s)} (\|b\|_{L^1(\Omega)} + \nu \|N^*(t)\|_{L^1(\Omega)}) ds\\
    &= Me^{\alpha(t-t_0)} \|N(\cdot,t_0)\|_{L^\infty(\Omega)} + \frac{c(\|b\|_{L^1(\Omega)} + \nu \|N^*(t)\|_{L^1(\Omega)})}{\delta}(1 - e^{-\delta(t-t_0)}) \\
    &\le Me^{\alpha(t-t_0)} \|N(\cdot,t_0)\|_{L^\infty(\Omega)} + \frac{c(\|b\|_{L^1(\Omega)} + \nu \|N^*(t)\|_{L^1(\Omega)})}{\delta}.
    \end{split} \end{equation}
    Now for any $\eta > 0$, taking $t_0 > 0$ large enough, we see from \eqref{eq:Nmax} that $$\|N^*(t)\|_{L^1(\Omega)} \le \eta + \frac{\|b\|_{L^1(\Omega)}}{\delta}, \,\,\,\,\,\, t \ge t_0.$$ Fixing this $t_0$, we can take $t_1$ large enough that $Me^{\alpha(t_1-t_0)}\|N(\cdot,t_0)\|_{L^\infty(\Omega)} < \eta$ as well. Inserting both of these bounds in \eqref{eq:Nbound2}, we have \begin{equation}\label{eq:Nbound3}
    N(x,t) \le \eta + \frac{c\|b\|_{L^1(\Omega)} + \nu(\|b\|_{L^1(\Omega)}/\delta + \eta)}{\delta}, \,\,\,\, t \ge t_1.
    \end{equation} This provides a uniform bound on $N(x,t)$ which holds when $t \ge t_1$. Note also that we can take $\eta > 0$ as small as desired, at the cost of increasing $t_1.$ 

    Thus, to ensure $\mu N - \nu \le 0$ for large $t$, it suffices to require that $$\nu > \mu \left( \eta + \frac{c\|b\|_{L^1(\Omega)} + \nu(\|b\|_{L^1(\Omega)}/\delta + \eta)}{\delta}\right)$$ or equivalently \begin{equation} \label{eq:MuAndNu} \nu \left(1 - \mu \left( \frac{\|b\|_{L^1(\Omega)}}{\delta^2} + \frac \eta \delta \right)\right) \ge \mu \eta + \frac{\mu c \|b\|_{L^1(\Omega)}}{\delta}.\end{equation} If $\mu$ is too large, this will be impossible since the left hand side above will be negative. However, supposing, as in the hypotheses of the theorem, that $\mu < \delta^2/\|b\|_{L^1}$, we can take $\eta$ small enough that $\mu < 1/(\|b\|_{L^1(\Omega)}/\delta^2 + \eta / \delta).$ Next, supposing that $\nu > \mu c \delta \|b\|_{L^1(\Omega)} / (\delta^2-\mu \|b\|_{L^1(\Omega)}),$ we can decrease $\eta$ again if necessary to ensure that \eqref{eq:MuAndNu} holds, and thus $\mu N(x,t) - \nu\le 0$ when $t \ge t_1.$ Then from \eqref{eq:Nstar}, for $t \ge t_1$, we have \begin{equation} \label{eq:NstarEqDecay}\partial_t N^* \le d^*\Delta N^* - \delta- N^*,\end{equation} so that, by lemma \ref{lem:bound}, $N^*$ decays uniformly to zero at an exponential rate as $t\to\infty$. By positivity of solutions, this implies exponential decay of $S^*,I^*,R^*$ for large time.

    The uniform bound on $N(x,t)$ given by \eqref{eq:Nbound3} also provides a uniform bound on $S,I,R$, meaning that all terms involving $S^*,I^*,R^*$ in \eqref{eq:allCompliant} decay exponentially for large time. Thus \eqref{eq:allCompliant} is asymptotic to \begin{equation} 
    \label{eq:asympCompDFE}
    \begin{split}
    (\partial_t -d\Delta)S &\le  b(x) - \beta(1-\alpha)^2SI - \delta S, \\
    (\partial_t -d\Delta)I &\le  \beta(1-\alpha)^2SI - (\gamma + \delta) I,\\
    (\partial_t -d\Delta)R &\le  \gamma I - \delta R,
    \end{split}
    \end{equation} Analyzing this system, it is straightforward to prove convergence of the solution to $(\tilde S,0,0)$ under the condition that $\R_0 < 1$, using methods similar to those in section \ref{sec:nonCompDFE}. Indeed, we see $$(\partial_t - d\Delta)(S-\tilde S) \le - \beta(1-\alpha)^2SI - \delta(S-\tilde S) \le -\delta(S-\tilde S). $$ Using an integrating factor and applying lemma \ref{lem:bound}, this shows that $S(x,t) \to \tilde S(x)$ uniformly as $t\to\infty$. In particular, for arbitrary $\eps > 0$, we can find $t_2 \ge t_1$ large enough that $$\abs{S(x,t)} \le \abs{\tilde S(x)} + \eps, \,\,\,\,\,\, t \ge t_2.$$ Then for $t \ge t_2$, $$(\partial_t-d\Delta)I \le (\beta(1-\alpha)^2(\tilde S + \eps) - \gamma)I.$$ By lemma \ref{lem:localComp}, since we are assuming $\R_0 < 1$, we have that $\lambda(\tilde S)<0$, where $\lambda(\tilde S)$ is the leading eigenvalue for equation \eqref{e.w08012}. By continuity, for sufficiently small $\eps > 0$, $\lambda(\tilde S + \eps) < 0.$ Allowing $\varphi$ to be the positive eigenfunction corresponding to $\lambda(\tilde S + \eps)$, and taking $A>0$ large enough that $I(x,t_2) \le A\varphi(x)$, an application of the comparison principle shows that $$I(x,t) \le A\varphi(x) e^{\lambda(\tilde S+\eps)(t-t_2)}, \,\,\,\,\,\, t \ge t_2,$$ whereupon $I(x,t) \to 0$ uniformly as $t\to\infty$. Finally, this shows that the $R$ equation is asymptotic to $$(\partial_t - d\Delta R) = -\delta R$$ which gives exponential decay of $R$ as well. Thus $(S,I,R,S^*,I^*,R^*) \to (\tilde S,0,0,0,0,0)$ uniformly as $t\ to \infty$ which concludes the proof of \textit{(i)}. 

Next we prove \textit{(ii)}. Let $\mathcal R_0 > 1$ and assume toward a contradiction that for any $\eps_0 > 0$, there is a positive solution of \eqref{eq:allCompliant} satisfying 
			 \begin{equation}
			 \limsup_{t\to\infty}\|(S,I,R,S^*,I^*,R^*)-(\tilde S,0,0,0,0,0)\|_{L^\infty(\Omega)}\le\epsilon_0.
			 \end{equation}	
		 Then in particular, $S(x,t) \to \tilde S(x)$ and $I^*(x,t), I(x,t), N^*(x,t) \to 0$ uniformly as $t \to \infty$, which means that the $I$ equation in \eqref{eq:allCompliant} is asymptotic to $$(\partial_t - d\Delta)I = \beta(1-\alpha)^2SI - \gamma I.$$ From here, the proof proceeds exactly as the proof of \textit{(ii)} in theorem \ref{thm:globalStabilityNonComp}: $\R_0 > 1$ implies $\lambda(\tilde S - \eps_0) > 0$ for sufficiently small $\eps_0$, which implies exponential growth of $I(x,t)$ for large time, contradicting our assumption. 

   We conclude that when $\R_0 > 1$, there exists $\eps_0 > 0$ such that any positive solution of \eqref{eq:allCompliant} satisfies
				\begin{equation}
				\limsup_{t\rightarrow \infty}
				\|(S,I,R,S^*,I^*,R^*)-(\tilde S,0,0,0,0,0)\|_{L^\infty(\Omega)}>\epsilon_0.
				\end{equation}
\end{proof}

\WW{\textit{Remark.} In the proof of theorem \ref{t.w11191}(i), the condition on $\nu$ in \eqref{eq:munubound} can be relaxed slightly to instead read \begin{equation}\label{eq:nuModified}\nu > \frac{c\mu \delta \|b\|_{L^1(\Omega)} - \delta_*^3}{\delta^2 - \mu \pnorm b 1 \Omega},\end{equation} for some $\delta_* \in [0,\delta).$ Following the proof through, this results in the bound $\mu N - \nu \le \delta^*$, and we achieve exponential decay of $N^*$ from \eqref{eq:NstarEqDecay}, with decay rate $\delta - \delta_*$. We note that if $\mu$ is sufficiently small, the numerator on the right hand side of \eqref{eq:nuModified} is negative. In this case, the condition is automatically satisfied, and we arrive at the result with no constraint on $\nu$. That is, if $\mu$ is small enough and $\xi = 1$ so that all newly introduced members of the population are compliant, then there is no possibility of endemic noncompliance, even in the absence of recovery from noncompliance.} 

\section{Simulation \& Discussion}\label{s5}

In this final section, we simulate our model using MATLAB, discuss results with emphasis on how the behavior of our model differs from that of vanilla SIR-type models. We note that the theorems above address the cases where either (1) the population is almost entirely noncompliant, wherein the effective reproductive ratio of the disease is closer to $\R_0^*$, or (2) the population remains almost entirely compliant, wherein the effective reproductive ratio of the disease is closer to the smaller value $\R_0$. However, over the course of the epidemic, the effective reproductive ratio will be changing: it should be some sort of weighted average of $\R_0$ and $\R_0^*$, depending on the portion of the population which is noncompliant. These effects are very difficult to capture analytically because they depend on the intermediate-time dynamics of the model. However, we can demonstrate the effects through simulation. 

To simulate the model, we use a semi-implicit finite difference scheme, wherein the diffusion is resolved implicitly, but the nonlinear terms are resolved explicitly. We perform simulations in several parameter regimes to demonstrate different features of the model.  For simplicity, we use a square domain $\Omega = (-5,5)^2$ for all simulations. In all cases, we set the natural birthrate to be constant $b(x) = b$. In figures \ref{fig:1}-\ref{fig:6}, all diffusion coefficients to be equal (we denote the mutual value $d$); we experiment with the varying the diffusion coefficient for the infectious populations in figures \ref{fig:7} and \ref{fig:8}. The parameters corresponding to the simulation which produced each of the below figures are listed in table \ref{tab:params}. The other key pieces of data are the initial conditions. Figures \ref{fig:1}-\ref{fig:2} correspond to the same simulation, and have a unique initial condition, whereas each of figures \ref{fig:3}-\ref{fig:6} correspond to different simulations, but each of figures \ref{fig:3}-\ref{fig:6} have the same initial condition. We specify these below as well. In all cases, we choose $R_0 = R_0^* = 0$ so that initially, there is no recovered population. Additionally, we only specify $S_0$ and $I_0$ and then set $S_0^* = S_0/20$ and $I_0^* = I_0/20$. In doing so, we are assuming that initially, roughly $5\%$ of the population is noncompliant. We emphasize that all of these parameter values and initial conditions are synthetic and were chosen simply to demonstrate the behavior of the model in different regimes.

\begin{table}[t]
\centering
\begin{tabular}{|c r r r r r|}
\hline
Parameter & Figs. 1-2 & Fig. 3 & Fig. 4 & Fig. 5 & Fig. 6 \\
\hline
   $\beta$  &  6  & 0.05 & 50 & 1 & 1 \\ 
  $\gamma$  &  1  & 1 & 1 & 1 & 1 \\ 
       $b$  &  0.02  & 0.02 & 0.02 & 0.02 & 0.02 \\ 
  $\delta$  &  0.001  & 0.001 & 0.001 & 0.001 & 0.001 \\ 
  $\alpha$  &  0.5  & 0.1 & 0.8 & 0.8 & 0.8 \\ 
     $\mu$  &  1  & 1 & 1 & 0.01 & 2 \\ 
     $\nu$  &  1  & 0 & 0 & 0.015 & 0.015 \\ 
     $\xi$  &  0.95  & 0 & 0 & 1 & 1 \\ 
       $d$  &  0.02  & 0.02 & 0.02 & 0.02 & 0.02 \\ 
\hline
\end{tabular}
\caption{Parameter values used for simulations displayed in each figure below.}
\label{tab:params}
\end{table}

\WW{While figures \ref{fig:1}-\ref{fig:2} demonstrate general observations regarding our model, it is of particular interest in figures \ref{fig:3}-\ref{fig:6} to demonstrate the results of theorems \ref{thm:globalStabilityNonComp} and \ref{t.w11191}, which make assumptions on the reproductive ratios $\mathcal R_0^*$ and $\mathcal R_0$. Because of this, it is convenient to quantify these reproductive ratios given the parameter values in those figures. Note that if the birth rate $b(x) = b$ is constant (as in all of our simulations), the steady-state solutions of \eqref{e.w07291} and \eqref{eq:compliantSeq} are constant. For exmaple, $\tilde S^*$ satisfies $$-d_{S^*} \tilde S^* = b - (\nu + \delta)\tilde S^*$$ along with Neumann boundary conditions, which has the unique solution $\tilde S^* = \frac{b}{\nu + \delta},$ and likewise one finds that $S^* = \frac b \delta$ is the unique steady-state solution of \eqref{eq:compliantSeq}. Because of this, the eigenvalue problems given by \eqref{e.w07293} and \eqref{e.w08012} are constant coefficient, and thus can be solved explicitly in the square domain $\Omega = (-5,5)\times (-5,5).$ Specifically, due to the Neumann boundary conditions, the eigenfunctions for each equation have the form $\phi_{k,\ell}(x,y) = \cos(k\pi x/5)\cos(\ell\pi y/5)$. The principle eigenvalues then correspond to $k=\ell=0$. These are given by \begin{equation} \label{eq:eigenvalues}\lambda^* = \frac{b\beta}{\nu + \delta} - (\gamma + \delta), \,\,\,\,\,\,\,\,\,\, \lambda = \frac{b\beta(1-\alpha)^2 }{\delta} - (\gamma+\delta)\end{equation} for \eqref{e.w07293} and \eqref{e.w08012}, respectively. Recall, by lemmas \ref{lem:localNonComp} and \ref{lem:localComp}, $\R_0^* >1$ if and only if $\lambda^* > 0$, and likewise for $\R_0$ and $\lambda$. Thus these eigenvalues gives us a manner of ensuring that we fall into the correct parameter regimes in order to demonstrate theorems \ref{thm:globalStabilityNonComp} and \ref{t.w11191}.}

The first simulation demonstrates something that is not necessary unique to our model, but is unique to spatial models, and is accented even further by our model. In this case, we set $S_0(x) = \text{exp}(-5\abs{x}^2)$ for $x \in (-5,5)^2$ so that the initial susceptible population is concentrated very strongly at the origin, whereas $I_0(x) = \frac{1}{20} \text{exp}(-5\abs{x - (3,3)}^2)$ for $x \in (-5,5)^2$, meaning that the initial infections are very strongly concentrated at $(3,3),$ and comprise roughly $ 5\% $ of the total population. Because of this, for small $t >0$, $S(\alpha I+I^*)$ and $S^*(\alpha I + I^*)$ are approximately zero, meaning very few new infections occur, and initially the total number of infections decreases. However, after enough time, the populations have diffused enough that there is more overlap causing more infections to occur, which results in a later spike in the total number of infections. We see this in figure \ref{fig:1}, where we plot the total portion of the infected population $$I_{\text{total}}(t) = \|(I+I^*)(t)\|_{L^1(\Omega)} / \|(S+I+R+S^*+I^*+R^*)(t)\|_{L^1(\Omega)}$$ as a function of time. Also plotted in figure \ref{fig:1} is the total portion of the noncompliant population: $$N^*_{\text{total}}(t) = \|(S^*+I^*+R^*)(t)\|_{L^1(\Omega)} / \|(S+I+R+S^*+I^*+R^*)(t)\|_{L^1(\Omega)}.$$ Note that for this simulation, the initial spike in infections occurs before a significant portion of the population becomes noncompliant. The first spike in infections then declines due to the decline in the susceptible population, which causes a decline in the effective reproductive ratio of the disease. However, when enough of the population becomes noncompliant, the effective reproductive ratio increases because noncompliant populations have a higher infection rate, which causes a second wave. This behavior then repeats. In figure \ref{fig:2}, we display snapshots of the infectious population $I(x,t)+I^*(x,t)$ at different times $t \ge 0$, where we see that the initial profile is concentrated near the point $(3,3)$. The infectious population then ``migrates" toward  the origin where the susceptible population is concentrated, and increases when it is sufficiently close.

\begin{figure}
    \centering
    \includegraphics[width=\textwidth]{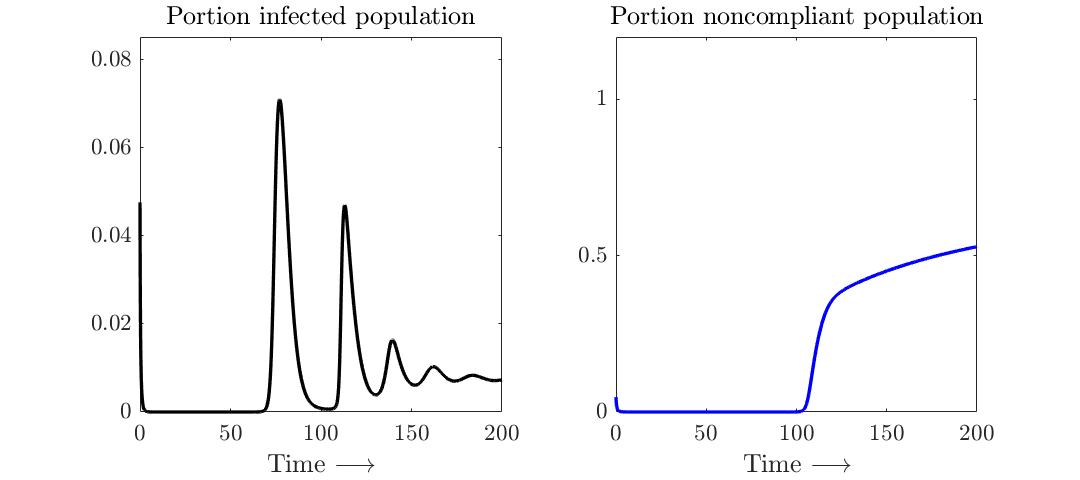}
    \caption{The infections initially decay, and then spike once the susceptible populations and infectious populations diffuse enough that they overlap. After the initial spike, the noncompliant population is large enough to increase the effective reproductive ratio and cause another spike. The infectious population for this simulation is plotted in figure \ref{fig:2}.}
    \label{fig:1}
\end{figure}

\begin{figure}
    \centering
    \includegraphics[width = 0.31\textwidth]{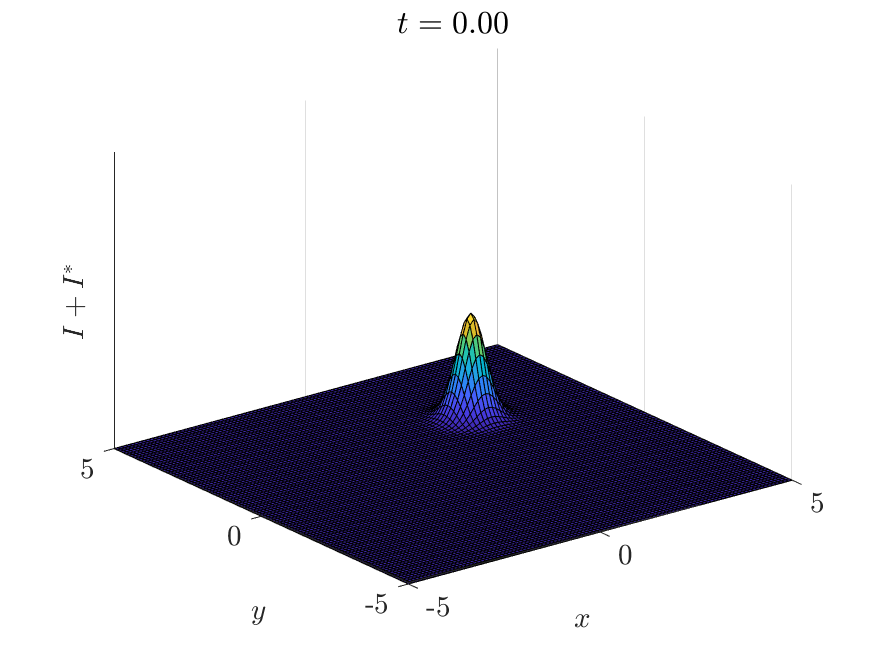}\,\includegraphics[width = 0.31\textwidth]{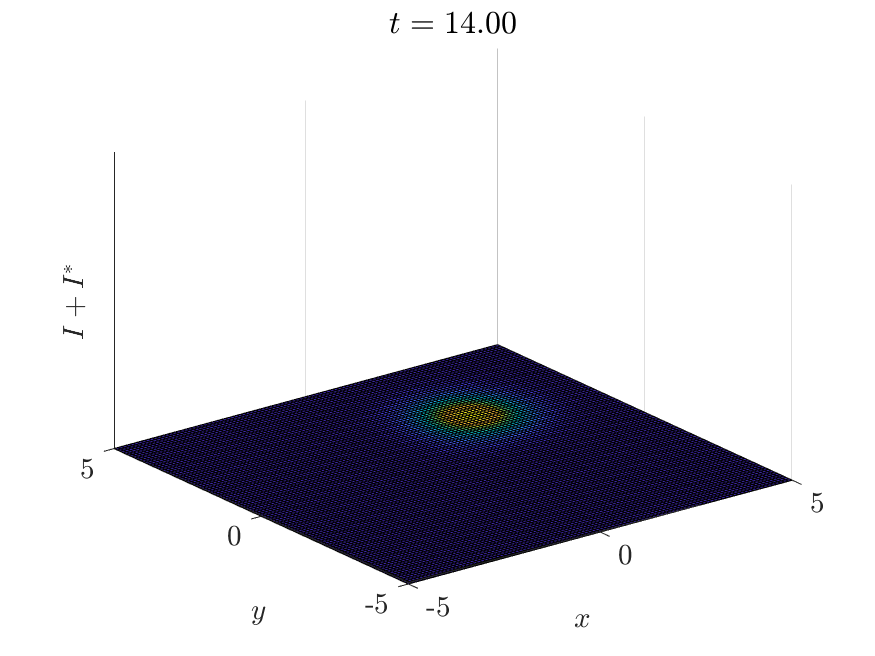}\,\includegraphics[width = 0.31\textwidth]{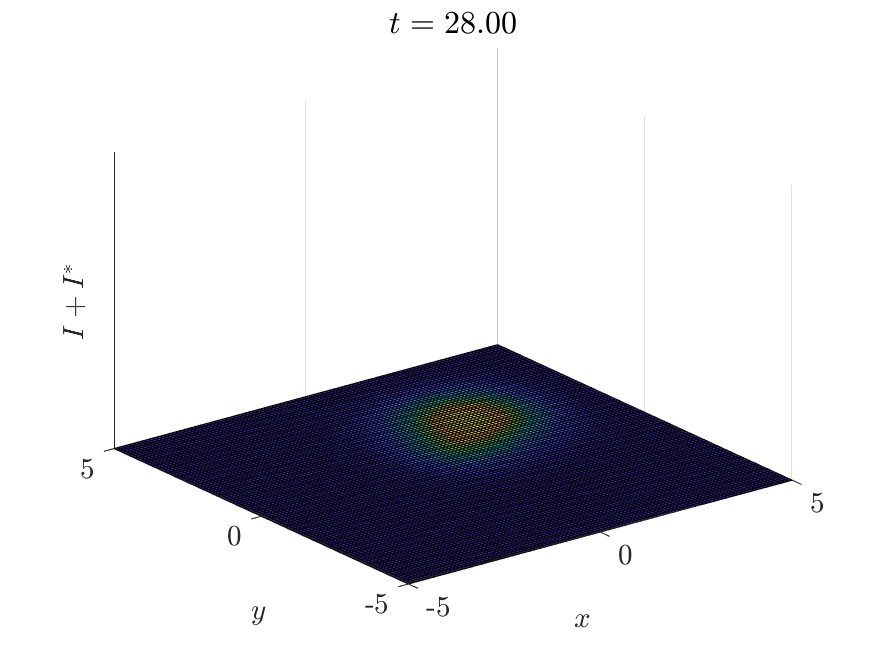}\\\includegraphics[width = 0.31\textwidth]{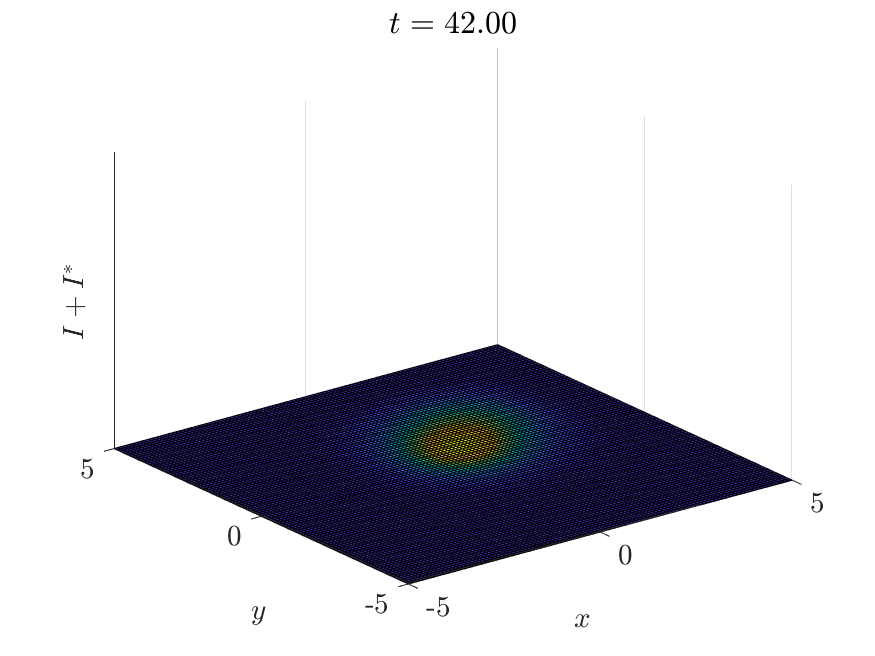}\,\includegraphics[width = 0.31\textwidth]{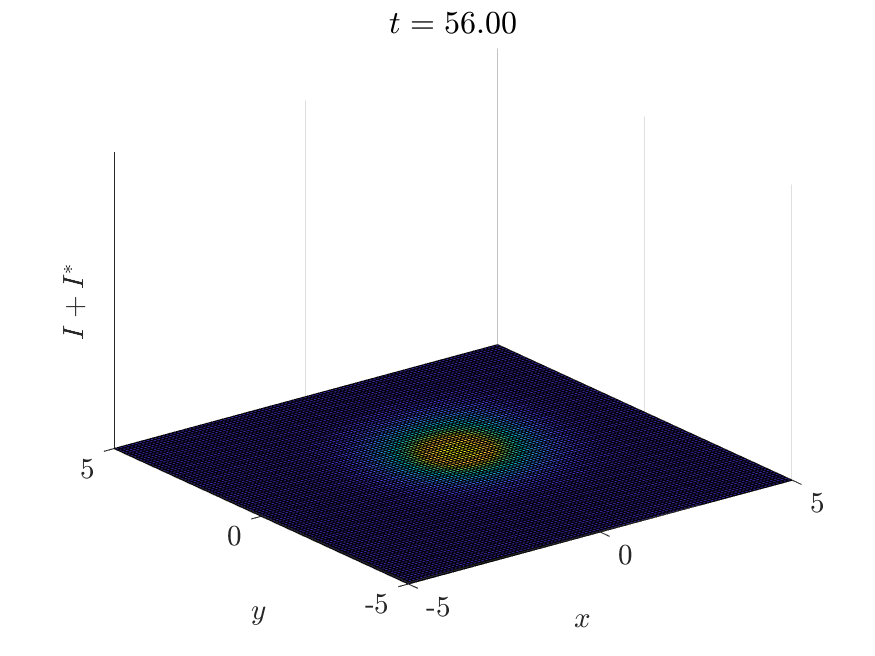}\,\includegraphics[width = 0.31\textwidth]{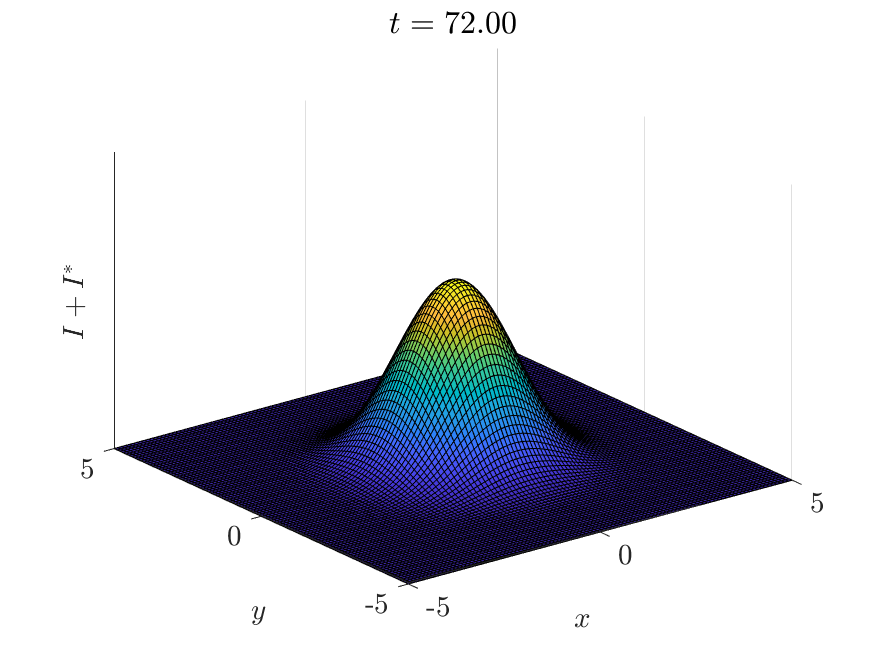}
    \caption{Snapshots of the infectious population $I+I^*$ as time $t$ increases. The infections are initially concentrated near $(3,3)$. They decline since there is very little susceptible population in this area. The infections then ``migrate" toward the origin, where the susceptible population is more concentrated, and increase once they are close enough.}
    \label{fig:2}
\end{figure}

In the ensuing figures, we let $S_0$ be a sum of four Gaussians, centered at different points around $(-5,5)$ and having different variances. These could be thought of as population centers like urban areas, which are much more densely populated than the rural areas surrounding them. Here $I_0 = S_0 / 100$, meaning that the infectious population is distributed identically to the susceptible population, but comprises only about $1 \% $ of the population. 

Figure \ref{fig:3} demonstrates theorem \ref{thm:globalStabilityNonComp}(i). In this simulation, \WW{the parameters ensure that $\lambda^* = \frac{b\beta}{\nu + \delta} - (\gamma + \delta) = - \delta < 0$, so that $\R_0^* < 1$. In this case, even when nearly the entire population is noncompliant, the infection rate is still too low to cause an outbreak.} Note that in this case, $\xi = 0$, so all newly introduced members of the population are noncompliant. As far as noncompliance goes, this is somewhat of a worst case scenario, and driving the infection rate down is the only means of preventing an infection. 

\begin{figure}
    \centering
    \includegraphics[width=\textwidth]{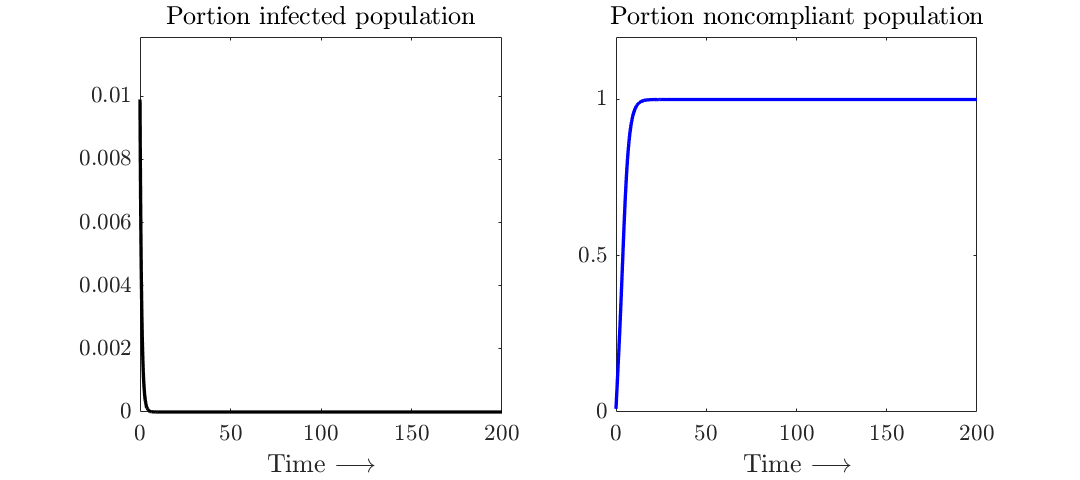}
    \caption{If the infection rate $\beta$ is small enough, the effective reproductive ratio will stay small even when the entire population becomes noncompliant, so that no outbreak occurs. This demonstrates theorem \ref{thm:globalStabilityNonComp}(i).}
    \label{fig:3}
\end{figure}

In figure \ref{fig:4}, we demonstrate theorem \ref{thm:globalStabilityNonComp}(ii). In this simulation, the infection rate $\beta = 50$ is very large, so that $\lambda^* = \frac{b\beta}{\nu +\delta} - (\gamma + \delta) \approx 1000 > 0$, meaning that $\R_0^* > 1$. In fact, even with an $80\%$ reduction in infectivity due to compliance ($\alpha = 0.8$), we have $\lambda = \frac{b\beta(1-\alpha)^2}{\delta} - (\gamma + \delta) \approx 40 >0$ so that $\R_0 > 1$. Thus, even at the outset of the simulation when a majority of the population is compliant, the infection still grows. Due to this, we see a swift increase in infections. The total portion of the infected population also does not tend to zero, but rather settles at some nonzero constant: approximately $1 \% $ of the total population is infected in the asymptotic regime. 

\begin{figure}
    \centering
    \includegraphics[width=\textwidth]{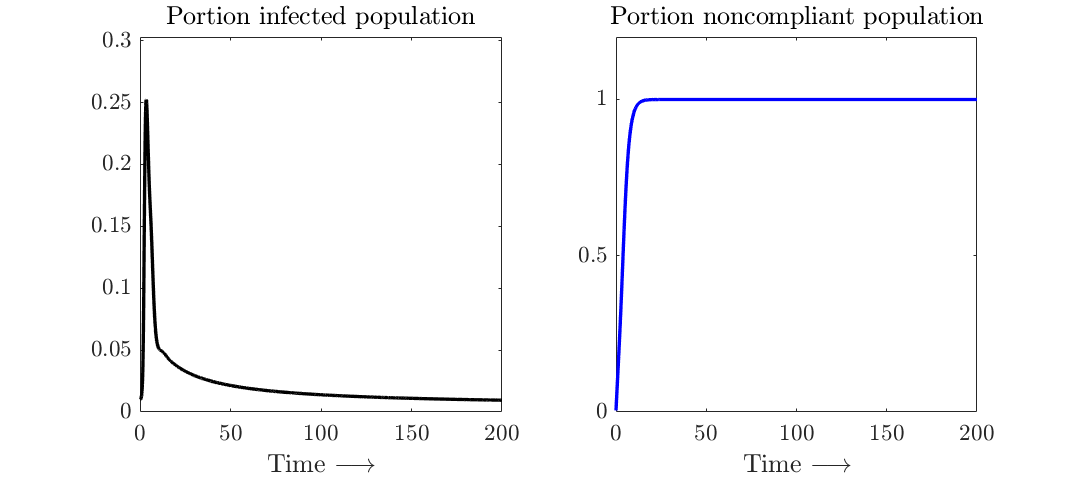}
    \caption{If the infection rate $\beta$ is large enough, the effective reproductive ratio is large even when the vast majority of the population is compliant, and the reproductive ratio only increases from there, so there is no hope of containing the disease and trending toward a disease free state. Note that the This demonstrates theorem \ref{thm:globalStabilityNonComp}(ii).}
    \label{fig:4}
\end{figure}

\WW{Figures \ref{fig:5} and $\ref{fig:6}$ demonstrate the dependence of the outcome of the epidemic on the noncompliance infection rate and recovery rate ($\mu$ and $\nu$ respectively), which is encoded in theorem \ref{t.w11191}. In these simulations, parameters are chosen so that $\lambda = \frac{b\beta(1-\alpha)^2}{\delta} - (\gamma + \delta) \approx -\frac 1 5< 0$, and thus $\R_0< 1$. However, $\lambda^* = \frac{b\beta}{\nu + \delta} - (\gamma + \delta) \approx \frac 1 3$, so $\R_0^* > 1$. In this case, preventing an outbreak of the disease would hinge upon the population remaining compliant.} This is seen in figures \ref{fig:5} and \ref{fig:6}. Note that the simulations for these figures hold all parameters constant except for $\mu$ and $\nu$. In figure \ref{fig:5}, $\mu$ is small relative to $\nu$, meaning that individuals become noncompliant at a slower rate, and once they become noncompliant, transfer back to compliance very quickly. In this case, the portion of the noncompliant population remains small for all time, and no outbreak occurs. However, in figure \ref{fig:6}, the roles are reversed: $\mu$ is large relative to $\nu$, meaning that the noncompliant populations grow more rapidly, and become compliant again more slowly. In this case, while there is initial decrease in the total portion of the infected population, after a large enough portion becomes noncompliant, there is an outbreak of the disease, as effective reproductive ratio approaches $\R_0^*$. This demonstrates the crucial dependence theorem \ref{t.w11191}(i) on $\mu$ and $\nu$. If one can guarantee that the population remains compliant (i.e., if $\mu$ is small enough and $\nu$ is large enough), one can achieve asymptotic stability of the disease free state, but this also demonstrates that with no such guarantee, an outbreak can occur. 

\begin{figure}
    \centering
    \includegraphics[width=\textwidth]{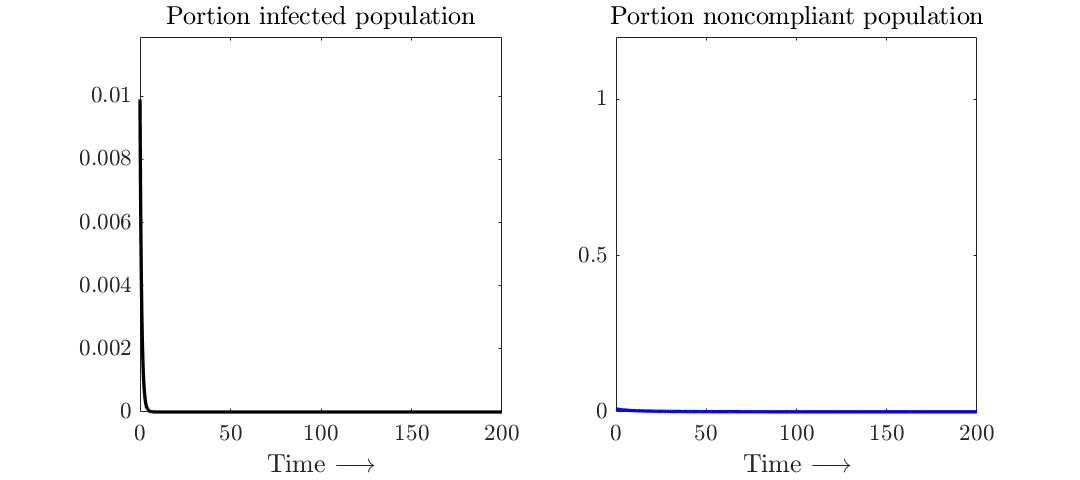}
    \caption{When $\mu$ is small and $\nu$ is large, the portion of the noncompliant population remains small for all time, meaning that the effective reproductive ratio of the disease remains small and no outbreak occurs.}
    \label{fig:5}
\end{figure}

\begin{figure}
    \centering
    \includegraphics[width=\textwidth]{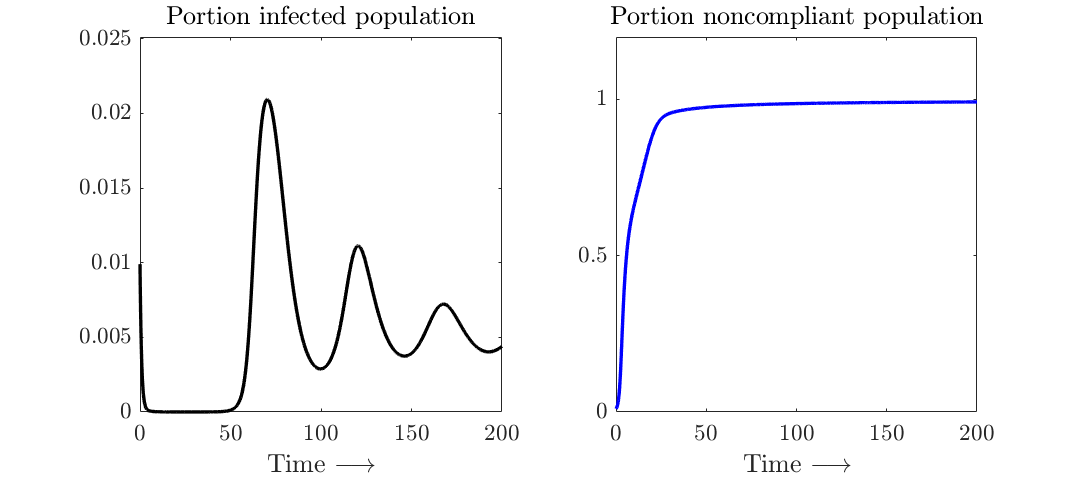}
    \caption{When $\mu$ is large and $\nu$ is small, the portion of the noncompliant populations grows, which increases the effective reproductive ratio of the disease so that an outbreak occurs.}
    \label{fig:6}
\end{figure}

\WW{Finally, for technical reasons, our theorems require assumptions regarding the diffusion coefficients, and to satisfy all of these assumptions simultaneously, it is simplest to consider the case where all diffusion coefficients are the same. However, in simulations we can vary these to empirically observe the behavior. Specifically, it is interesting to observe the behavior as $d_I$ and $d_{I^*}$ are alternately made very large or very small, since these are the coefficients that appear in \eqref{e.w08013},\eqref{e.w08014} and\eqref{e.w07293},\eqref{e.w07294}, respectively. For all the ensuing simulations, we use the same parameter values as in figure \ref{fig:6}, except that in figures \ref{fig:7} and \ref{fig:8} we vary the value of $d_I$ and $d_{I^*}$ while leaving all other diffusion coefficients fixed at $d = 0.02,$ and in figure \ref{fig:9}, we vary the value of $d_{S^*}, d_{I^*}$ and $d_{R^*}$ while leaving the other diffusion coefficients fixed at $d = 0.02.$}

\begin{figure}
\centering
\includegraphics[width=\textwidth]{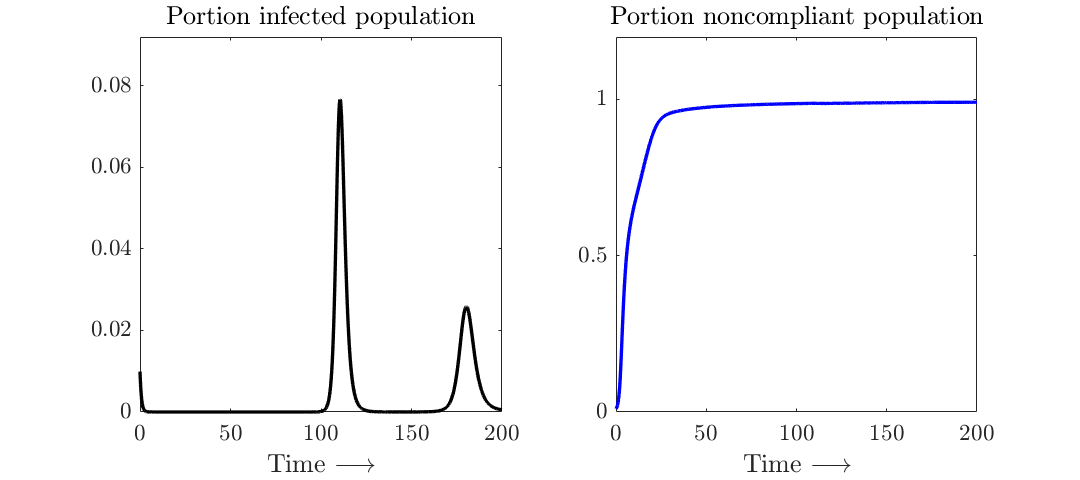}\\
\includegraphics[width=\textwidth]{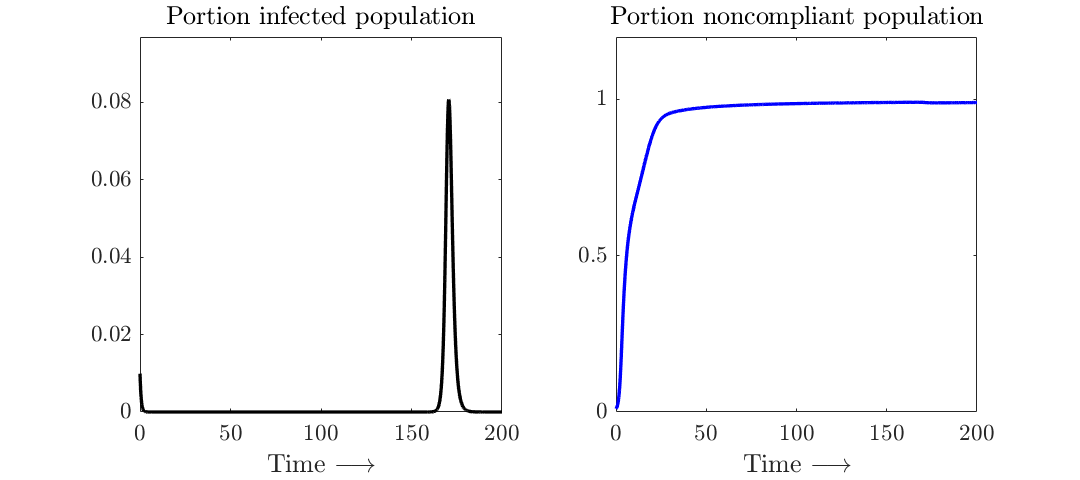}
\caption{Simulations with the same parameter values in figure \ref{fig:6}, except that $d_I=d_{I^*} = 1$ (top) or $d_I = d_{I^*} = 3$ (bottom). Empirically, when the diffusion coefficients for the infectious populations are increased, the initial outbreak is forestalled, and the infection peaks become sharper.}
\label{fig:7}
\end{figure}

\WW{In figure \ref{fig:6}, we note that the first outbreak occurs at roughly time $t = 50$. In figure \ref{fig:7}, when $d_I,d_{I^*}$ are increased to $1$ (representing a fiftyfold increase), the first outbreak does not occur until roughly time $t = 100$, and when $d_I,d_{I^*}$ are increased to $3$, the first outbreak does not occur until roughly $t = 150$. The peaks also become much sharper when $d_I,d_{I^*}$ are larger. By contrast, in figure \ref{fig:8}, when $d_I,d_{I^*}$ are decreased to $0.0004$ (representing a fiftyfold decrease), the first peak again occurs at roughly $t =50$; it may be impossible for the first outbreak to occur before this time, since this calibration requires a large portion of the population to be noncompliant in order to produce an outbreak. However, in this case, the peaks are mollified to the point that they bleed together, and peaks occur in more rapid succession. Thus empirical evidence seems to imply that increasing the diffusion coefficients for the infectious populations while leaving others fixed forestalls the epidemic while making infection peaks sharper. These observations could be of interest to policy-makers, and speak somewhat to the interesting intermediate-time dynamics displayed by reaction-diffusion systems which can be very difficult to quantify. }

\begin{figure}
\centering
\includegraphics[width=\textwidth]{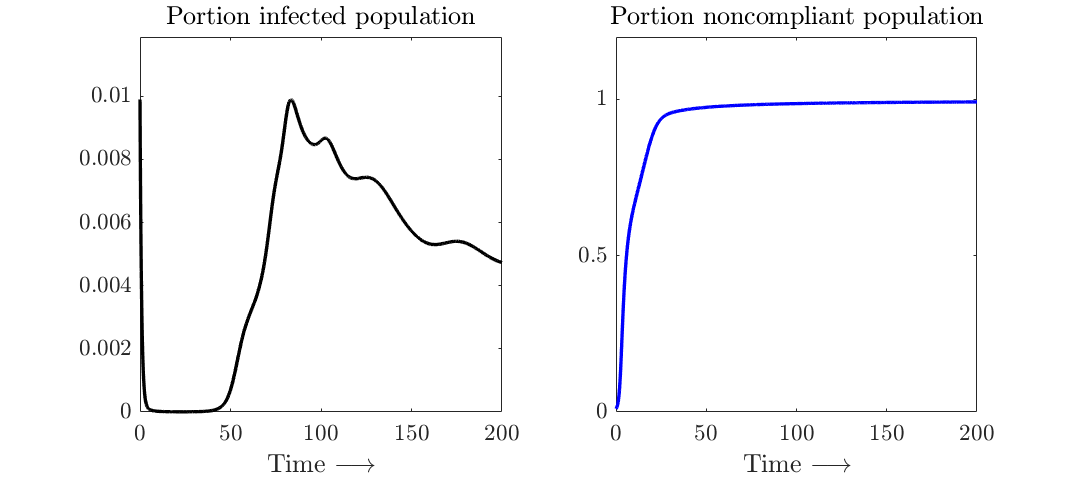}
\caption{Simulation with the same parameter values in figure \ref{fig:6}, except that $d_I=d_{I^*} = 0.0004$. Decreasing the diffusion coefficients for the infectious population causes the infection peaks to blur together. The initial outbreak occurs at roughly the same time as in figure \ref{fig:6}, because the noncompliant population must become sufficiently large for an outbreak to occur. }
\label{fig:8}
\end{figure}

\WW{Our last set of simulations in figure \ref{fig:9} demonstrate the effects of increasing the diffusion coefficients for the noncompliant populations ($d_{S^*},d_{I^*},d_{R^*})$ while leaving others fixed. Recall, if $d_S = d_I = d_R$ and $d_{S^*} = d_{I^*} = d_{R^*}$, then the compliant population $N = S+I+R$ and noncompliant population $N^* = S^* + I^* + R^*$ satisfy an SIS style system given by \eqref{eq:N},\eqref{eq:Nstar}, where $N$ is viewed as the ``susceptible" population and $N^*$ is viewed as the ``infectious" population. With this interpretation, we have heuristics provided by the SIS literature. For example, the authors of \cite{allen} analyze the behavior of the steady-state endemic solution to a similar SIS system in the limit as the diffusion coefficient of the susceptible population goes to zero while the diffusion coefficient of the infectious population remains fixed (or equivalently, the limit as the infectious diffusion coefficient goes to infinity while the susceptible diffusion coefficient remains fixed). They prove that in this limit, the steady-state endemic solution for the infectious population tends to zero. Their system is slightly different so the result does not directly apply to \eqref{eq:N},\eqref{eq:Nstar}, but it gives reason to expect that as $d_{S^*}, d_{I^*}, d_{R^*}$ grow, the large time limit of $N^*$ should go to zero. The simulations support this conclusion. In \ref{fig:9}, we use the same parameters as in figure \ref{fig:6}, except successively increase $d_{S^*},d_{I^*}, d_{R^*}$ from $0.02$ to $0.05$, then $0.1$, then $1$, then $5$. Because the increase is small at first (top images in figure \ref{fig:9}), the effect to the total noncompliant population is subtle, but one can discern the effect because the slight decrease in noncompliant population causes the epidemic to progress more slowly. In the bottom two images of figure \ref{fig:9}, the effect of enlarging the diffusion coefficient is more pronounced: when $d_{S^*} = d_{I^*} = d_{R^*} = 1$ or $5$, the total noncompliant population appears to settle at a much smaller value, and because of this, no epidemic occurs within the displayed time frame.}

\begin{figure}[t!]
\centering
\includegraphics[height = 0.221\textheight]{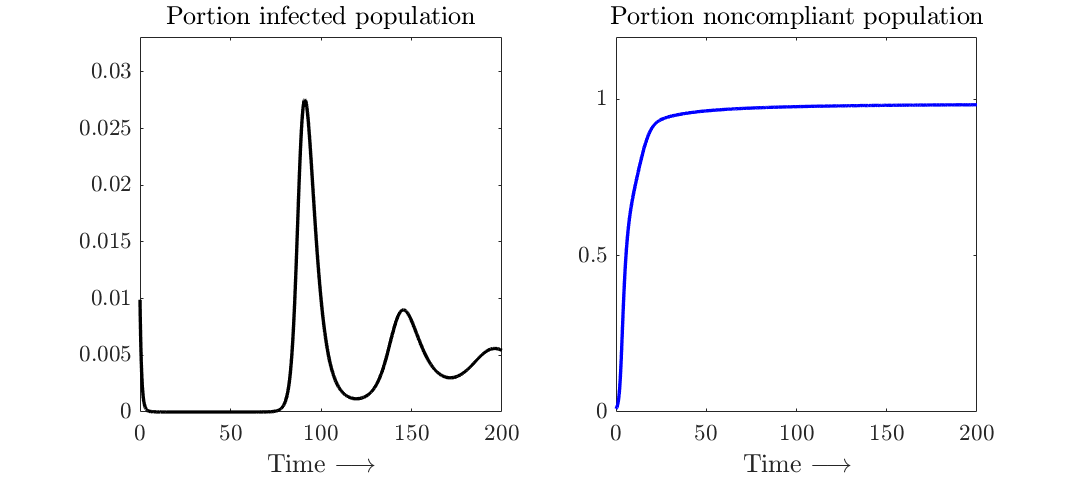}\\
\includegraphics[height = 0.221\textheight]{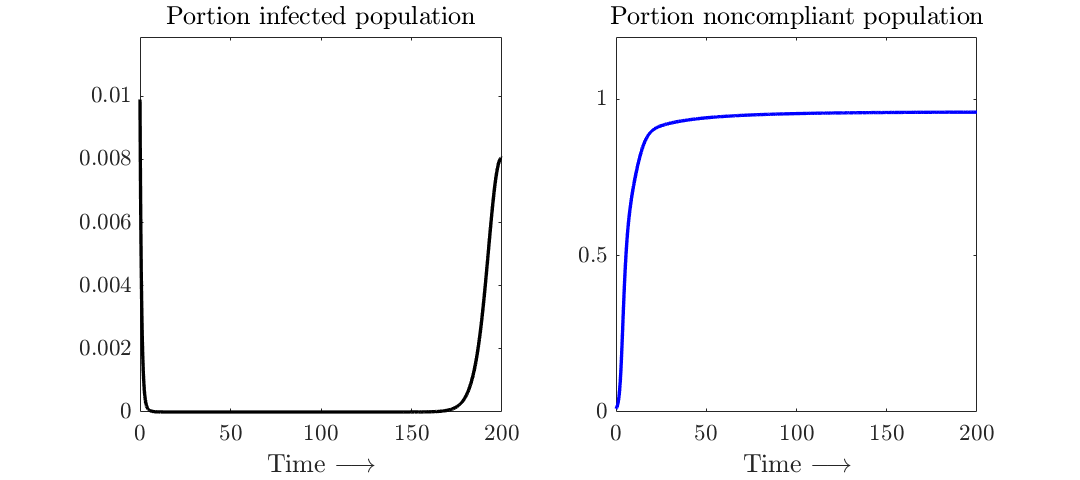}\\
\includegraphics[height = 0.221\textheight]{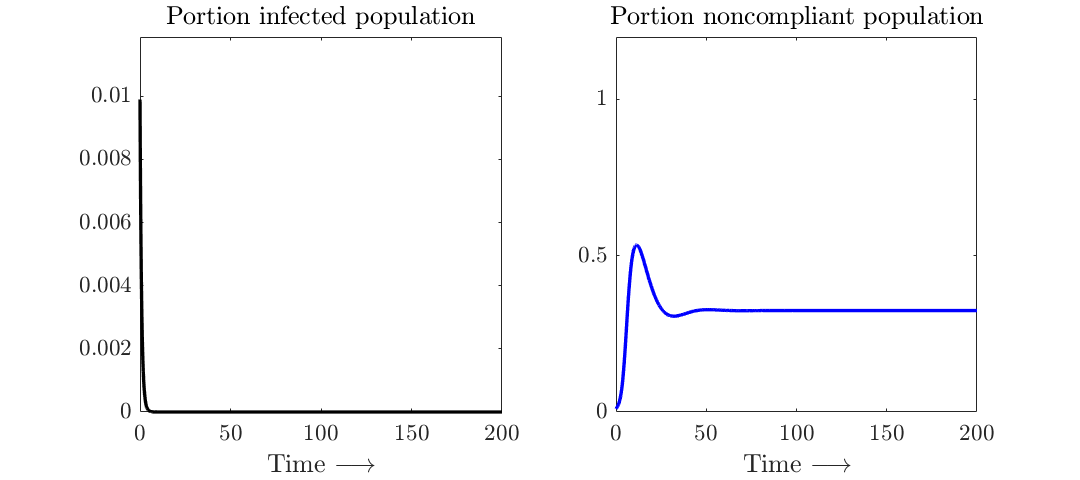}\\
\includegraphics[height = 0.221\textheight]{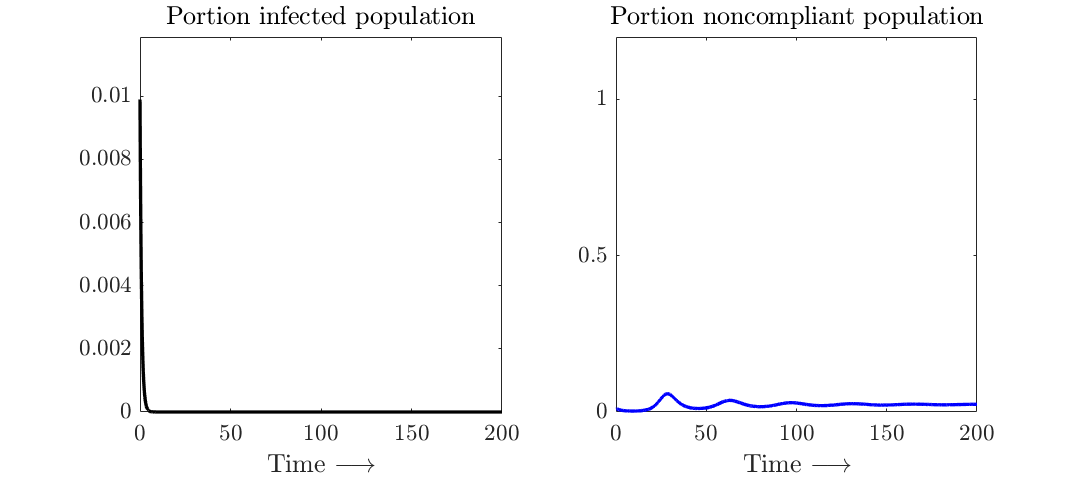}\\
\caption{Parameters are the same as in figure \ref{fig:6} except $d_{S^*},d_{I^*},d_{R^*}$ increase to $0.05$ (top), $0.1$ (second), $1$ (third), $5$ (bottom). As these coefficients increase, the total noncompliant population settles at a smaller value, which agrees with intuition provided by \cite{allen}.}
\label{fig:9}
\end{figure}

\section{Conclusion \& Future Directions} \label{sec:conclusion}

In this work, we present a reaction-diffusion SIR-type epidemic model, wherein noncompliance with prevention measures spreads via mass-action parallel to a disease. A key assumption is that the disease spreads more quickly among populations which are not compliant with prevention measures. We present proofs of global existence for our system, as well as $\R_0$ analysis and asymptotic behavior in different parameter regimes.

We propose four future directions for work along these lines. \WW{First, in theorems \ref{thm:globalStabilityNonComp} and \ref{t.w11191}, we prove global stability of disease free states under the alternate assumptions that $\xi = 0$ or $\xi = 1$ (so that any newly introduced members to the population are noncompliant or compliant, respectively). It would be very interesting to prove similar results in the case that $\xi \in (0,1)$, whereupon the steady-state system is the following coupled nonlinear elliptic equations: \begin{equation} 
\begin{split}
-d_S\Delta \tilde S &= \xi b(x) - \mu \tilde S\tilde S^* + \nu \tilde S^* - \delta \tilde S, \\
-d_{S^*}\Delta \tilde S^* &= (1-\xi)b(x) + \mu \tilde S\tilde S^* - \nu \tilde S^* - \delta \tilde S^*.
\end{split}
\end{equation} In this case, if $d_S \neq d_{S^*}$, it is unclear whether solutions even exist. However, if $d_S = d_{S^*}$, then one can solve by first considering $\tilde \Sigma = \tilde S + \tilde S^*$ (which satisfies a simple linear elliptic equation), and having resolved $\tilde \Sigma$, both $\tilde S$ and $\tilde S^*$ each satisfy their own decoupled nonlinear elliptic equation, for which we have existence of solutions. However, in this case, the derivation of the principle eigenvalue and reproductive ratio is no longer so simple. One could still define the reproductive ratio as the spectral radius of a certain elliptic operator as in \cite{WZ}, and achieve local stability, but the definition is then less quantitative (that is, it is more difficult to see precisely how the reproductive ratio depends on relevant parameters), and  it is less clear how to arrive at conditions for global stability. A full exploration of these questions could prove very interesting and illuminating.}

Second, similar work to this is carried out in a network-theoretic setting in \cite{CP2}. It may be of interest to develop other types of epidemic models---for example, agent-based or self-exciting point process models---which incorporate human behavior in similar ways. Analysis and synthesis of these different types of models could elucidate the different implications of social contagion theory in epidemiology. Third, there has been recent interest (even before the onset of COVID-19), in coupling within-host and between-host models for infectious diseases \cite{multiscale0,multiscale1,multiscale2,multiscale3}. Incorporating human behavior into these models in a similar manner to what we suggest here may result in very high fidelity modeling of a pandemic. Finally, our work elucidates different facets of an epidemic given our assumptions regarding the manner in which human behavior is incorporated. However, this work is entirely qualitative. To push toward real-world utility, a more data-driven study which incorporates parameter estimation would likely be necessary.

\section*{Acknowledgments}
CP is supported in part by NSF DMS-1937229 through the Data Driven Discovery Research Training Group at the University of Arizona. WW is supported in part by an AMS-Simons travel grant, and would like to thank Chris Henderson for many useful comments and remarks. Both authors were supported by a postdoctoral collaborative research grant through the University of Arizona Department of Mathematics. \WW{The authors would like to thank two anonymous reviewers for their very detailed and helpful comments, especially regarding the remark after the proof of theorem 4.4 and for suggestions regarding the simulations and results section.}

\clearpage

\bibliographystyle{plain}
\bibliography{references}
\end{document}